\documentclass[11pt]{article}
\usepackage{amsfonts,amsthm}
\usepackage{color}
\usepackage{amsbsy}
\usepackage{dsfont}
\usepackage{hyperref}
\usepackage{amsfonts}
\usepackage{latexsym}

\usepackage{amsmath,amssymb,mathdots,array,booktabs,graphicx,mathrsfs}
\usepackage{algpseudocode,algorithm}
\usepackage[T1]{fontenc}
\usepackage{color}

\newcommand{\bsmat}{\left[\begin{smallmatrix} }
\newcommand{\esmat}{\end{smallmatrix}\right] }

\newcommand{\tn}{\mathbb{T}_n}
\newcommand{\dtau}{\mathbb{D}_{\tau_n}}
\newcommand{\rnn}{\mathbb{R}^{n\times n}}
\newcommand{\cnn}{\mathbb{C}^{n\times n}}
\newcommand{\on}{\mathbb{O}_n}
\newcommand{\un}{\mathbb{U}_n}
\newcommand{\sn}{\mathbb{S}_n}
\newcommand{\hn}{\mathbb{H}_n}
\newcommand{\grn}{\mathbb{GL}(n,\mathbb{R})}
\newcommand{\gcn}{\mathbb{GL}(n,\mathbb{C})}
\newcommand{\an}{\mathbb{A}_n}
\newcommand{\wn}{\mathbb{W}_n}
\newcommand{\PI}{PI}

\DeclareMathOperator{\T}{T}
\DeclareMathOperator{\HH}{H}
\DeclareMathOperator{\diag}{diag}
\DeclareMathOperator{\card}{card}

\DeclareMathOperator{\rank}{rank}
\DeclareMathOperator{\subspan}{span}

\DeclareMathOperator{\Bdiag}{Bdiag}
\DeclareMathOperator{\OffBdiag}{OffBdiag}

\DeclareMathOperator{\subspace}{subspace}

\newtheorem{theorem}{Theorem}[section]
\newtheorem{lemma}{Lemma}[section]
\newtheorem{definition}{Definition}[section]

\newtheorem{remark}{{\sc Remark}}[section]
\newtheorem{example}{Example}[section]

\numberwithin{equation}{section}

\title{Solving General Joint Block Diagonalization Problem via  Linearly Independent Eigenvectors of a  Matrix Polynomial
\thanks{This research was supported by NSFC under grants 11671023, 11421101 and 11301013.}}

\author{
Yunfeng Cai\thanks{
LMAM \& School of Mathematical Sciences,
Peking University, Beijing, 100871, China,
{\tt yfcai@math.pku.edu.cn}
} \quad
Guanghui Cheng\thanks{
School of Mathematical Sciences, University of Electronic Science and Technology of China, Chengdu, Sichuan 611731, P. R. China,
{\tt ghcheng@uestc.edu.cn}}
\quad
Decai Shi\thanks{
Beijing International Center for Mathematical Research, Peking University, Beijing, 100871, P. R. China,
{\tt decaishi@gmail.com}
}
}
\date{\today}

\begin{document}

\maketitle

\begin{abstract}
In this paper, we consider the exact/approximate general joint block diagonalization (GJBD) problem of
a matrix set $\{A_i\}_{i=0}^p$ ($p\ge 1$), where a nonsingular matrix $W$ (often referred to as diagonalizer) needs to be found such that
the matrices $W^{\HH}A_iW$'s are all exactly/approximately block diagonal matrices
with as many diagonal blocks as possible.
We show that the diagonalizer of the exact GJBD problem can be given by $W=[x_1, x_2, \dots, x_n]\Pi$,
where $\Pi$ is a permutation matrix,
$x_i$'s are eigenvectors of the matrix polynomial $P(\lambda)=\sum_{i=0}^p\lambda^i A_i$,
satisfying that $[x_1, x_2, \dots, x_n]$ is nonsingular,
and the  geometric multiplicity of each $\lambda_i$ corresponding with $x_i$ equals one.
And the equivalence of all solutions to the exact GJBD problem is established.
Moreover, theoretical proof is given to show why the approximate GJBD problem can be solved similarly to the exact GJBD problem.
Based on the theoretical results, a three-stage method is proposed and numerical results show the merits of the method.

\vskip 2mm

\noindent {\bf Key words.}
general joint block diagonalization, matrix polynomial, tensor decomposition.

\vskip2mm

\noindent {\bf MSC.}  15A69, 65F15.
\end{abstract}

\section{Introduction}\label{sec:intro}
The problem of joint block diagonalization of matrices (also called simultaneously block diagonalization problem),
is a particular decomposition of a third order tensor in block terms  \cite{de2008decompositions}, \cite{de2008decompositions2}, \cite{de2008decompositions3}, \cite{nion2011tensor}.
Over the past two decades, such a decomposition has found many applications
in independent component analysis (e.g., \cite{cardoso1998multidimensional}, \cite{de2000fetal},
\cite{theis2005blind}, \cite{theis2006towards}) and semidefinite programming (e.g., \cite{bai2009exploiting}, \cite{de2007reduction}, \cite{de2010exploiting}, \cite{gatermann2004symmetry}).
For example, in blind source separation (BSS),
people aim to recover source signals from the observed mixtures,
without knowing either the distribution of the sources or the mixing process \cite{choi2005blind}, \cite{comon2010handbook}, \cite{hyvarinen2004independent}.
Different assumptions on the source signals lead to different models and methods.
Typically, there are three cases:
first, the source signals are mutually statistically  independent,
the mixing system can be estimated by joint diagonalization (JD),
e.g., JADE \cite{cardoso1993blind}, eJADE \cite{moreau2001generalization}, SOBI \cite{belouchrani1997blind} and Hessian ICA \cite{theis2004new}, \cite{yeredor2000blind};
second, there are several groups of signals,
in which components from different groups are mutually statistically independent
and statistical dependence occurs between components in the same group
(known as multidimensional BSS or group BSS),
the mixing system can be estimated by joint block diagonalization (JBD) \cite{cardoso1998multidimensional}, \cite{theis2005blind};
third, the number of groups and the size of each group are unknown,
the mixing system can be estimated by general joint block diagonalization (GJBD).

To proceed, in what follows we introduce some definitions and notations, then formulate the JBD problem and the GJBD problem mathematically.
\begin{definition}
We call $\tau_n=(n_1,n_2,\dots,n_t)$ a {\em partition} of positive integer $n$ if
$n_1,n_2,\dots,n_t$ are all positive integers and their sum is $n$, i.e., $\sum_{i=1}^t n_i=n$.
The integer $t$ is called the {\em cardinality} of the partition $\tau_n$,
denoted by $\card(\tau_n)$.
The set of all partitions of $n$ is denoted by $\tn$.
\end{definition}

\begin{definition}
Given a partition $\tau_n=(n_1,n_2,\dots,n_t)\in\tn$,
for any $n$-by-$n$ matrix $A$,
define its block diagonal part and off-block diagonal part associated with $\tau_n$ as
  \begin{align*}
    \Bdiag_{\tau_n}(A)=\diag(A_{11},A_{22},\dots,A_{tt}),\quad   \OffBdiag_{\tau_n}(A)=A-\Bdiag_{\tau_n}(A),
\end{align*}
respectively, where $A_{ii}$ is $n_i$-by-$n_i$ for $i=1,2,\dots, t$.
A matrix $A$ is referred to as a $\tau_n$-block diagonal matrix if $\OffBdiag_{\tau_n}(A)=0$. The set of all $\tau_n$-block diagonal matrices is denoted by $\dtau$.
\end{definition}

Let $\sn$, $\hn$, $\on$, $\un$, $\grn$ and $\gcn$
denote the $n\times n$ matrix set of real symmetric matrix, complex Hermitian matrix, real orthogonal matrix,
complex unitary matrix, real nonsingular matrix and complex nonsingular matrix, respectively.
Let $\an=\sn$, $\hn$, $\rnn$, or $\cnn$,
$\wn=\on$, $\un$, $\grn$, or $\gcn$.
Then the JBD problem and the GJBD problem  can be formulated as follows.

\noindent\textbf{The JBD problem.} \quad Given a matrix set $\mathcal{A}=\{A_i\}_{i=0}^p$
with $A_i\in\an$,
and a partition $\tau_n=(n_1,n_2,\dots,n_t)\in\tn$.
Find a matrix $W=W(\tau_n)\in\wn$ such that $W^{\star} A_i W\in\dtau$ for $i=0,1,\dots, p$,
i.e.,
\begin{equation}\label{eq:nujbd}
W^{\star} A_i W=\diag(A_{i}^{(11)},A_i^{(22)},\dots, A_{i}^{(tt)}), \quad \mbox{for}\quad i=0,1,\dots, p,
\end{equation}
where $A_{i}^{(jj)}$ is $n_j$-by-$n_j$ for $j=1,2,\dots,t$,
the symbol $(\cdot)^{\star}$ stands for the transpose of a real matrix or the conjugate transpose of a complex matrix.

%
%

\noindent\textbf{The GJBD problem.}\quad
Given a matrix set $\mathcal{A}=\{A_i\}_{i=0}^p$
with $A_i\in\an$.
Find a partition $\tau_n'=(n_1',n_2',\dots,n_t')\in\tn$ and a matrix $W=W(\tau_n')\in\wn$
such that
\begin{equation*}
\card(\tau_n')=\max\{\card(\tau_n) \, \big|\,
\mbox{there exists a $W=W(\tau_n)$ which solves JBD}\}.
\end{equation*}

The transformation matrix $W$ is often referred to as a {\em diagonalizer}.
The (G)JBD problem is called symmetric/Hermitian (G)JBD  if $\an=\sn/\hn$;
exact/approximate (G)JBD  if \eqref{eq:nujbd} is satisfied exactly/approximately;
orthogonal/ non-orthogonal (G)JBD  if $\an=\rnn$ and $\wn=\on/\grn$;
similarly, 
unitary/non-unitary (G)JBD  if $\an=\cnn$ and $\wn=\un/\gcn$.

In practical applications,
the matrices $A_i$'s are usually constructed from empirical data,
as a result, the exact JBD problem has no solutions
and the exact GJBD problem has only trivial solution $((n), I_n)$.
Consequently, the approximate (G)JBD problem is considered instead.
For the JBD problem, it is  natural to formulate it as a constrained optimization problem $C(W)=\min$,
where $C(\cdot)$ is a cost function used to measure the off-block diagonal parts of $A_i$'s,
$W$ is a diagonalizer in certain feasible set.
Different cost functions and feasible sets  together with various optimization methods lead to many numerical methods.
Since this paper mainly concentrates on algebraic methods,
we will not list the detailed literature on the optimization methods,
we refer the readers to  \cite{chabriel2014joint}, \cite{de2009survey}, \cite{tichavsky2017non} and references therein.
For the GJBD problem, one needs to minimize the off-block diagonal parts of $A_i$'s
and maximize the number of diagonal blocks simultaneously,
it is difficult to formulate it as a simple optimization problem that can be easily solved.
By assuming that the GJBD problem shares the same local minima with the JD problem, the GJBD problem 
is simply solved with a JD algorithm,
followed by a permutation, which is used to reveal the block structure  \cite{abed2004algorithms}, \cite{theis2006towards}.

Without good initial guesses, optimization methods may suffer from slow convergence,
or converge to degenerate solutions.
Algebraic methods, on the other hand, are able to find a solution in finite steps
with predictable computational costs.
And even if the solutions returned by algebraic methods  are ``low quality'',
they are usually good initial guesses for optimization methods.
In current literature, the algebraic methods for the GJBD problem fall into two categories:
one is based on matrix $\ast$-algebra
(see  \cite{de2011numerical}, \cite{maehara2010numerical}, \cite{maehara2011algorithm},  \cite{murota2010numerical}, for the orthogonal GJBD problem and a recent generation \cite{cai2017algebraic} for the non-orthogonal GJBD problem),
the other is based on matrix polynomial (see \cite{cai2015matrix} for the Hermitian GJBD problem).
In the former category,
a null space of a linear operator needs to be computed,
which requires $\mathcal{O}(n^6)$ flops,
thus, for problems with large $n$,
such an approach will be quite expensive for both storage and computation.
In this paper, we will focus on the latter category,
which we will show later that it only requires $\mathcal{O}(n^3)$ flops.

As the results in this paper is an extension of those in \cite{cai2015matrix},
in what follows,  we summarize some related results therein.
For a Hermitian matrix set $\{A_i\}_{i=0}^p$,
the corresponding matrix polynomial is constructed as
${P}_{\mathcal{A}}(\lambda)=\sum_{i=0}^p\lambda^i A_i$.
Assuming that ${P}_{\mathcal{A}}(\lambda)$ is regular and has only simple eigenvalues,
and using the spectral decomposition of the Hermitian matrix polynomial, theoretically,
it is shown that the column vectors of the diagonalizer of the exact Hermitian GJBD problem of $\{A_i\}_{i=0}^p$ can be given by 
$n$ linearly independent eigenvectors (in a certain order) of ${P}_{\mathcal{A}}(\lambda)$ (\cite[Corollary 3.5]{cai2015matrix});
all solutions to the Hermitian GJBD problem are equivalent, i.e.,
all solutions are unique up to block permutations and block diagonal transformations (\cite[Theorem 3.8]{cai2015matrix}).
Therefore, one can solve the Hermitian GJBD problem by finding $n$ linearly independent eigenvectors $x_1, x_2, \dots, x_n$ of ${P}_{\mathcal{A}}(\lambda)$,
followed by determining a permutation $\Pi$ via revealing the block diagonal structure of $\Pi^{\T}[x_1, \dots, x_n]^{\HH}A_i [x_1, \dots, x_n]\Pi$ (MPSA-II).
Numerically, it is shown that MPSA-II,
though designed to solve the exact Hermitian GJBD problem,
is able to deal with the approximate Hermitian GJBD problem to some extent. 
However, the approach in \cite{cai2015matrix} suffers from the following three disadvantages:
first, the proofs are difficult to follow
if the readers are unfamiliar with the spectral decomposition of a matrix polynomial;
second, there is no theoretical proof to show why the approach is applicable for the approximate Hermitian GJBD problem;
third,
the approach can not be used to solve general ($A_i$'s are not necessarily Hermitian) GJBD problem directly.
\footnote{
By constructing a Hermitian matrix polynomial
$\widehat{P}_{\mathcal{A}}(\lambda)\triangleq\lambda^{2p+1}(A_p+A_p^{\HH})+
\lambda^{2p} \imath (A_p-A_p^{\HH})+\dots+\lambda (A_0+A_0^{\HH})+\imath (A_0-A_0^{\HH})$,
one can still follow the approach in \cite{cai2015matrix} to solve the general GJBD problem of $\{A_i\}_{i=0}^p$,
but the degree of $\widehat{P}_{\mathcal{A}}(\lambda)$ is $2p+1$,
almost twice as many as the degree of ${P}_{\mathcal{A}}(\lambda)$.}
In this paper, we try to give a remedy.
For a matrix set $\mathcal{A}=\{A_i\}_{i=0}^p$,
we still construct the matrix polynomial as $P_{\mathcal{A}}(\lambda)=\sum_{i=0}^p\lambda^i A_i$.
Let $x_1, x_2, \dots, x_n$ be $n$ linearly independent eigenvectors of ${P}_{\mathcal{A}}(\lambda)$.
Under the assumption that the geometric multiplicities of the corresponding eigenvalues equal one
(much weaker than that ``${P}_{\mathcal{A}}(\lambda)$ is regular and all eigenvalues are simple''), we show that
the diagonalizer of
the exact GJBD problem can be written as $X\Pi=[x_1, x_2, \dots, x_n]\Pi$,
where $\Pi$ is a permutation matrix;
all solutions to the exact GJBD problem are equivalent.
The proofs of these results are easy to follow,
without using the spectral decomposition of a matrix polynomial.
Furthermore, using perturbation theory,
we  give a theoretical proof for using $X\Pi$ as the diagonalizer for the approximate GJBD problem.
Lastly, a three-stage method, which is modified from MPAS-II in \cite{cai2015matrix}, is proposed.
Numerical examples show that the proposed method is effective and efficient.

The rest of the paper is organized as follows.
In section~\ref{sec:pre}, we give some preliminary results on matrix polynomials and motivations for using a matrix polynomial to solve the GJBD problem.
In section~\ref{sec:theory}, the main results are presented.
Numerical method and numerical examples are given in sections~\ref{sec:method} and \ref{sec:eg}, respectively.
Finally, we present some concluding remarks in section~\ref{sec:conclusion}.

\bigskip

\noindent{\bf Notations.}\qquad
The imaginary unit $\sqrt{-1}$ is denoted by $\imath$.
For a matrix $A=[a_{ij}]$, $|A|$, $\|A\|_2$ and  $\|A\|_F$ denote $[|a_{ij}|]$, the 2-norm and Frobinius norm, respectively.
The eigenvalue sets of a square matrix $A$ and a matrix polynomial $P(\lambda)$ are denoted by $\lambda(A)$ and $\lambda(P)$, respectively.
The MATLAB convention is adopted to access the entries of vectors and matrices.
The set of integers from $i$ to $j$ inclusive is $i : j$.
For a matrix $A$, its submatrices $A{(k:\ell,i:j)}$, $A{(k:\ell,:)}$, $A{(:,i:j)}$ consist of intersections of
row $k$ to row $\ell$ and column $i$ to column $j$,
row $k$ to row $\ell$ and all columns,
all rows and column $i$ to column $j$, respectively.

\section{Preliminary and motivation}\label{sec:pre}

A matrix polynomial of degree $p$ is defined as
\begin{align}\label{mpp}
P(\lambda)\triangleq\lambda^p A_p+\lambda^{p-1}A_{p-1}+\dots+A_0,
\end{align}
where the coefficient matrices $A_i$'s are all $n$-by-$n$ matrices and $A_p\ne 0$.
A scalar $\lambda$ is called an eigenvalue of $P(\lambda)$ if
$\det(P(\lambda))=0$.
A nonzero vector is called the corresponding eigenvector if $P(\lambda)x=0$.
Such $\lambda$ together with $x$ are called an eigenpair of $P(\lambda)$,
denoted by $(\lambda,x)$.

The polynomial eigenvalue problem (PEP) $P(\lambda)x=0$ is equivalent to a generalized
eigenvalue problem (GEP) $(\lambda M + N)u=0$,
where
\begin{align}\label{uxl}
u=u(x,\lambda)&\triangleq [x^{\T}, \lambda x^{\T}, \dots, \lambda^{p-1}x^{\T}]^{\T},
\end{align}
$L$ and $M$ are some $np$-by-$np$ matrices.
Such a transformation is called {\em linearization}.
Linearizations are not unique, among which, the commonly used one can be given by
\begin{align}\label{mn}
M=\begin{bmatrix}
I& 0 &0&\cdots&0\\
0&I & 0 &\cdots&0\\
\vdots&\ddots&\ddots&\ddots&\vdots\\
0 & \cdots & 0& I& 0\\
0 & \cdots & 0 &0 & A_p
\end{bmatrix},\quad
N=\begin{bmatrix}
0&-I&0&\cdots&0\\
0&0&-I&\ddots&\vdots\\
\vdots&\vdots&\ddots&\ddots&0\\
0 & 0& \cdots& 0 & -I\\
A_0& A_1 &\cdots& A_{p-2}& A_{p-1}
\end{bmatrix}.
\end{align}
For more linearizations and some structure preserving linearizations for structured matrix polynomials, we refer the readers to \cite{higham2006symmetric}, \cite{mackey2006vector}, \cite{mehrmann2002polynomial}.
The eigenvalues and eigenvectors of PEP can be obtained via those of GEP, vice versa.

\bigskip

At first glance, it seems that the matrix polynomial is not related to the GJBD problem at all. But in fact, they are closely related.
Let us consider the GJBD problem of three matrices $\{A_0,A_1,A_2\}$ with
\begin{align*}
A_0=\begin{bmatrix}  7 &8 & 9\\ 4 & -12 & -8\\ 5 & -4 & 7\end{bmatrix},\quad
A_1=\begin{bmatrix} -8 & 8 & 8\\ -4 & 4  & 0\\ -4 & 12 & 0 \end{bmatrix},\quad
A_2=\begin{bmatrix} 5 & 0 & 3\\ -8 & 4 & -4\\ -5 & 4 & 1 \end{bmatrix}.
\end{align*}
The eigenvector matrix $X$ and the eigenvalue matrix $T$ of the PEP
$ (\lambda^2 A_2 +\lambda A_1 +A_0)x=0$ can be given by
\begin{align*}
X&=[x_1, x_2, \dots, x_6]\\
&=\begin{bmatrix}
  -0.1690 & -0.5774 &  0.3981 + 0.4094\imath &  0.3981 - 0.4094\imath & -0.5774   & 0.4904 \\
  -0.9710& -0.5774  & 0.1108 + 0.5792\imath  & 0.1108 - 0.5792\imath & -0.5774 & -0.7205\\
  -0.1690&  0.5774&  0.3981 + 0.4094\imath   &0.3981 - 0.4094\imath  & 0.5774 &   0.4904\\
\end{bmatrix},\\
T&=\diag(\lambda_1,\lambda_2,\dots,\lambda_6)\\
&=\diag( -3.5830, 3.0000,  -0.5283 + 1.3793\imath,  -0.5283 - 1.3793\imath,1.0000, 0.6396).
\end{align*}
Let $W=X{(:,[2, 1, 6])}$, then
\begin{align*}
W^{\T} A_0 W&=\begin{bmatrix}
 4.0000 &  -0.0000  &  0.0000\\
   -0.0000 & -10.5142  & -8.3243\\
    0.0000  & -13.1077 &   0.5034\end{bmatrix},\\
W^{\T} A_1 W&=  \begin{bmatrix} -5.3333 &  -0.0000  & -0.0000\\
    0.0000  &  6.2833    &7.4702\\
   -0.0000  & -6.8800  & -4.5381\end{bmatrix},\\
W^{\T} A_2 W&=\begin{bmatrix}
    1.3333  & -0.0000  &  0.0000\\
    0.0000   & 2.5726   & 8.6677\\
    0.0000 &  -0.8992   & 5.8646\end{bmatrix}.
\end{align*}
By calculations, we can show that the $2$-by-$2$ blocks in $W^{\T}A_0W$, $W^{\T}A_1W$ and $W^{\T}A_2W$ can not be simultaneously diagonalized.
Therefore, $((1,2),W)$ is a solution to the GJBD problem of $\{A_0, A_1, A_2\}$.
This example shows that the  GJBD problem can be indeed solved via linearly independent eigenvectors of
a matrix polynomial, in the proceeding sections, we will present the theoretical proofs.

\bigskip

Given a matrix set $\mathcal{A}=\{A_i\}_{i=0}^p$, the matrix polynomial $P_{\mathcal{A}}(\lambda)$
is constructed in a particular order of the matrices in $\mathcal{A}$.
However, the matrices are not ordered in any way for the GJBD problem of $\mathcal{A}$.
Later, we will see that the results in this paper do not
depend on such an order.
%
%
So it suffices to show the results for only one matrix polynomial, say $P_{\mathcal{A}}(\lambda)$.

\section{Main results}\label{sec:theory}

In this section, we give our main results for the exact and approximate GJBD problems in subsections~\ref{subsec:exact} and \ref{subsec:approx}, respectively.

\subsection{On exact GJBD problem}\label{subsec:exact}
In this subsection, we first characterize the diagonalizer $W$,
then show the equivalence of the solutions.
\begin{theorem}\label{thm1}
Given a matrix set $\mathcal{A}=\{A_i\}_{i=0}^p$.
Let $X=[x_1,x_2,\dots, x_n]$, $T=\diag(\lambda_1,\lambda_2,\dots, \lambda_n)$,
where $(\lambda_j,x_j)$ for $j=1,2,\dots, n$ are $n$ eigenpairs of  $P_{\mathcal{A}}(\lambda)=\sum_{i=0}^{p}\lambda^i A_i$.
Assume that $X$ is nonsingular and the  geometric multiplicities of $\lambda_1,\lambda_2,\dots, \lambda_n$ all equal one.
If $(\tau_n,W)$ solves the GJBD problem,
then there exist a permutation matrix $\Pi$ and
a nonsingular matrix $D\in\dtau$ such that $WD=X\Pi$,
i.e., $({\tau}_n,X\Pi)$  also solves the GJBD problem.
\end{theorem}

\begin{proof}
As $(\tau_n,W)$ is a solution to the GJBD problem,
we have
\begin{align}\label{eqB}
D_{i}\triangleq W^{\HH}A_iW=\diag(A_{i}^{(11)},\dots,A_{i}^{(tt)})\in\dtau \mbox{ for } i=0,1,\dots,p.
\end{align}
Since $(\lambda_j, x_j)$ for $j=1,2,\dots,n$ are eigenpairs of $P_{\mathcal{A}}(\lambda)$,
we also have
\begin{align}\label{matpolyeq}
A_p XT^{p} +A_{p-1}XT^{p-1}+\dots + A_0 X=0.
\end{align}
Pre-multiplying  \eqref{matpolyeq} by $W^{\HH}$ and using \eqref{eqB}, we get
\begin{align}\label{dyt}
D_{p} W^{-1}XT^{p}+D_{p-1}W^{-1}XT^{p-1}+\dots +D_0W^{-1}X=0.
\end{align}
Let $Y=W^{-1}X=[y_1,y_2,\dots, y_n]$, then
$(\lambda_j,y_j)$ for $j=1,2,\dots, n$ are eigenpairs of
$P_{\mathcal{D}}(\lambda)=\sum_{i=0}^p \lambda^i D_i$,
and the geometric multiplicities of $\lambda_j$'s, as eigenvalues of $P(\mathcal{D})$,
all equal one.
Denote $\mathcal{A}_j=\{A_i^{(jj)}\}_{i=0}^p$,
we know that for each $\lambda_j$, it belongs to a unique $\lambda({P}_{\mathcal{A}_k})$
since $\lambda(P_{\mathcal{D}})=\cup_{j=1}^t\lambda({P}_{\mathcal{A}_j})$,
and the geometric multiplicities of $\lambda_j$'s all equal one.
Let $n_1',n_2',\dots, n_t'$ be the numbers of $\lambda_j$'s in $\lambda({P}_{\mathcal{A}_1}),\lambda({P}_{\mathcal{A}_2}),\dots,\lambda({P}_{\mathcal{A}_t})$,
respectively.
Then there exists a permutation matrix $\Pi$ such that
\[
W^{-1}X\Pi=\bsmat {Y}_{11} & \dots &{Y}_{1t}\\
\vdots& \ddots &\vdots\\
{Y}_{t1} &\dots & {Y}_{tt}\esmat,\quad
\Pi^{\T}T\Pi=\diag({T}_1,\dots,{T}_t),
\]
where ${Y}_{jk}\in\mathbb{C}^{n_j\times n_k'}$
for $1\le j,k\le t$, $T_j\in\mathbb{C}^{n_j' \times n_j'}$, $\lambda(T_j)\subset \lambda({P}_{\mathcal{A}_j})$ for $j=1,2,\dots,t$.
The assumption that the  geometric multiplicities of $\lambda_1,\lambda_2,\dots, \lambda_n$ all equal one implies that
$\lambda(T_j)\cap\lambda(\mathcal{P}_{\mathcal{A}_k})=\emptyset$ for $j\ne k$, therefore,
by \eqref{dyt}, we have $Y_{jk}=0$ for $j\ne k$.
Thus, for any $1\le k\le t$, it holds that
\begin{align*}
n&=\rank(X\Pi)=\rank(Y_{kk})+\rank(\diag(Y_{11},\dots,Y_{k-1,k-1},Y_{k+1,k+1},\dots,Y_{tt}))\\
&\le\min\{n_k,n_k'\}+\min\{n-n_k,n-n_k'\}\\
&=\min\{n_k,n_k'\}+n-\max\{n_k,n_k'\}\le n.
\end{align*}
Then it follows that $\min\{n_k,n_k'\}=\max\{n_k,n_k'\}$, and hence $n_k=n_k'$.
Thus, 
$D\triangleq W^{-1}X\Pi\in\dtau$ is a nonsingular $\tau_n$-block diagonal matrix.
The conclusion follows.
\end{proof}

Based on Theorem~\ref{thm1}, we can solve the exact GJBD problem by finding $n$ linearly independent eigenvectors $X$ of $P_{\mathcal{A}}(\lambda)$,
then determining a permutation $\Pi$ by revealing the block structure of $\Pi^{\T}X^{\HH} A_i X \Pi$.
We will discuss the details in section~\ref{sec:method}.

\bigskip

Let $\tau_n=(n_1,n_2,\dots,n_t)\in\tn$, $(\tau_n, W)$ be a solution to the GJBD problem,
then $(\tau_n \Pi_t, W D \Pi)$ also solves the GJBD problem,
where $\Pi_t\in\mathbb{R}^{t\times t}$ is a permutation matrix,
$\Pi\in\rnn$ is a  block permutation matrix corresponding  with $\Pi_t$,
which can be obtained by replacing the ``1'' and ``0'' elements in $i$th row of $\Pi_t$
by a permutation matrix of order $n_i$ and zero matrices of right sizes, respectively,
$D\in\dtau$ is nonsingular.
For two solutions $(\tau_n, W)$, $(\hat{\tau}_n,\widehat{W})$ to the GJBD problem,
we say that $(\hat{\tau}_n,\widehat{W})$ is {\em equivalent} to $(\tau_n, W)$
if there exist a permutation matrix $\Pi_t$ and a nonsingular matrix $D\in\dtau$ such that
$(\hat{\tau}_n,\widehat{W})=(\tau_n \Pi_t, W D \Pi)$,
where $\Pi$ is the block permutation matrix corresponding  with $\Pi_t$.

Next, we show that all solutions to the GJBD problem are equivalent, under mild conditions.

\begin{theorem}\label{thm2}
Given a matrix set $\mathcal{A}=\{A_i\}_{i=0}^p$.
Let $\mathcal{B}=\{B_i\}_{i=0}^q$ be a matrix set such that $\subspan\{B_0,B_1,\dots,B_q\}=\subspan\{A_0,A_1,\dots,A_p\}$,
\footnote{The space spanned by several matrices is defined as $\subspan\{A_0,A_1,\dots,A_p\}=\{\sum_{i=0}^p \alpha_i A_i\; | \; [\alpha_0,\alpha_1,\dots,\alpha_p]^{\T}\in\mathbb{C}^{p+1}\}$.}
$P_{\mathcal{B}}(\lambda)=\sum_{i=0}^q \lambda^i B_i$.
If $P_{\mathcal{B}}(\lambda)$ has $n$ eigenpairs $(\lambda_j,x_j)$ for $j=1,2,\dots, n$  such that
$X=[x_1,x_2,\dots, x_n]$ is nonsingular and the geometric multiplicities of $\lambda_1,\lambda_2,\dots, \lambda_n$ all equal  one,
then all solutions to the GJBD problem of $\mathcal{A}$ are equivalent.
\end{theorem}

\begin{proof}
 First, by Theorem~\ref{thm1},
for any two solutions $(\tau_n,W)$, $(\hat{\tau}_n,\widehat{W})$ to the GJBD problem of $\mathcal{B}$,
there exist two permutation matrices $\Pi$ and $\widehat{\Pi}$
and two nonsingular matrices $D\in\dtau$ and $\widehat{D}\in\mathbb{D}_{\hat{\tau}_n}$
such that $WD=X\Pi$ and $\widehat{W}\widehat{D}=X\widehat{\Pi}$.
Let $H=\sum_{i=0}^q |X^{\HH}B_iX|$, $\widetilde{H}=[\tilde{h}_{ij}]$ with
\begin{align*}
\tilde{h}_{ij}=
\begin{cases}
1,\quad \mbox{if $i\ne j$ and $h_{ij}\ne 0$},\\
0,\quad \mbox{otherwise,}
\end{cases}
\end{align*}
where $h_{ij}$ is the $(i,j)$ entry of $H$.
On one hand, using the fact that $(\tau_n,W)$ and $(\hat{\tau}_n,\widehat{W})$ are both solutions to the GJBD problem,
we know that
$\Pi^{\T}\widetilde{H}\Pi\in\dtau$ and $\widehat{\Pi}^{\T}\widetilde{H}\widehat{\Pi}\in\mathbb{D}_{\hat{\tau}_n}$,  and $\card(\tau_n)=\card(\hat{\tau}_n)$.
On the other hand,
notice that $\widetilde{H}$ is a vertex-adjacency matrix of a graph $G$,
and $G$ has $t$ components $G_1,G_2,\dots,G_t$, and $G_i$ is connected with $n_i$ vertices for $i=1,2,\dots,t$,
where $n_i$ is the $i$th entry of $\tau_n$.
No matter how we relabel the vertices of $G$,
the collection of the numbers of vertices of all connected components of $G$ remains invariant.
Therefore, there exists a permutation matrix $\Pi_t$ of order $t$
such that $\hat{\tau}_n=\tau_n\Pi_t$ and $\widetilde{\Pi}=\Pi^{\T}\widehat{\Pi}$ is the block permutation matrix corresponding with $\Pi_t$.
Then it follows that
\[
\widehat{W}=X\widehat{\Pi}\widehat{D}^{-1}=(X\Pi D^{-1}) D (\Pi^{\T}\widehat{\Pi})\widehat{D}^{-1}
=WD\widetilde{\Pi}\widehat{D}^{-1}=WD(\widetilde{\Pi}\widehat{D}^{-1}\widetilde{\Pi}^{\T})\widetilde{\Pi}.
\]
In other words, all solutions to the GJBD problem of $\mathcal{B}$ are equivalent. Second, notice that for a partition $\tau_n$,
$W$ is a diagonalizer of the JBD problem of $\mathcal{A}$ if and only if
$W$ is a diagonalizer of the JBD problem of $\mathcal{B}$,
since $\subspan\{B_0,B_1,\dots,B_q\}=\subspan\{A_0,A_1,\dots,A_p\}$.
The conclusion follows immediately.
\end{proof}

\bigskip

Theorem~\ref{thm2} implies that the solution to the GJBD problem does not depend on the choices of the matrix polynomials,
neither the choices of $n$ linearly independent eigenvectors.

\subsection{On approximate GJBD problem}\label{subsec:approx}

In this subsection, 
we show that the solution to the approximate GJBD problem can also be written in the form $(\tau_n, X\Pi)$,
where $\tau_n$ is some partition of $n$, $X$ is a nonsingular matrix whose columns are eigenvectors of $P_{\mathcal{A}}(\lambda)$,
$\Pi$ is some permutation matrix.
The following two lemmas are needed for the proof.

%
%
\begin{lemma}\label{lem2}
Let $x$, $y\in\mathbb{C}^n$ be two nonzero vectors, then
$$\min_{t\in\mathbb{C}}\|x-t y\|_2=\|x\|_2\sin\angle(x,y),$$
where $\angle(x,y)$ is the angle between the subspaces spanned by $x$ and $y$, respectively.
\end{lemma}

The proof of Lemma~\ref{lem2} is simple, we will omit it here.
The following lemma is rewritten from  Theorem 4.1 in \cite{yuji2016eigenvector}.

\begin{lemma}\label{lem:yuji}
Let $(\hat{\lambda}, \hat{x})$ be an approximate eigenpair of $P(\lambda)=\sum_{i=0}^p\lambda^i A_i$ with residual $P(\hat{\lambda})\hat{x}\ne 0$.
Suppose $L(\lambda)=\lambda M + N$ is a linearization of $P(\lambda)$ given by \eqref{mn},
and $L(\lambda)$ has a generalized Schur form
\begin{align*}
Q^{\HH}MZ=\begin{bmatrix}\alpha_1 & \ast \\ 0 & M_1\end{bmatrix},
\quad
Q^{\HH}NZ=\begin{bmatrix}\beta_1 & \ast \\ 0 & N_1\end{bmatrix},
\quad
\lambda_1=-\frac{\beta_1}{\alpha_1},
\end{align*}
in which $M_1$, $N_1$ are both upper-triangular. 
Then the eigenvector $x_1$ of $P(\lambda)$ corresponding with $\lambda_1$ satisfies
\begin{align}\label{sin1}
\sin\angle(\hat{x􏰋},x_1) &\le
\frac{g\|P(\hat{\lambda})\hat{x}\|_2}{\sqrt{\sum_{i=0}^{p-1}|\hat{\lambda}|^{2i}}\|\hat{x}\|_2},
\end{align}
where
\begin{align}\label{gap}
g=g(\lambda_1,\hat{\lambda}; P(\lambda))\triangleq\frac{1}{\|(\hat{\lambda}M_1+N_1)^{-1}\|_2}.
\end{align}
\end{lemma}

%

Define a matrix set $\mathcal{W}$ as
\begin{align}\label{setw}
\mathcal{W}\triangleq \{W=[w_1,w_2,\dots, w_n]\; \big| \; \det(W)\ne 0, \, \|w_i\|_2=1 \mbox{ for } i=1,2,\dots, n\}.
\end{align}

Now we are ready to present the third main theorem.

\begin{theorem}\label{thm3}
Given a matrix set $\mathcal{A}=\{A_i\}_{i=0}^p$.
For a partition $\tau_n=(n_1,\dots,n_t)$,
assume that there exists a matrix $W\in\mathcal{W}$ such that
\begin{align}\label{offbwaw}
\|E_i\|_2\le \mu\|W^{\HH}A_i W\|_2,
\end{align}
where $E_i=\OffBdiag_{\tau_n}(W^{\HH} A_iW)$,
$\mu\ge 0$ is a parameter.
Let $(\lambda_j,x_j)$ for $j=1,2,\dots, n$ be $n$ eigenpairs
of $P_{\mathcal{A}}(\lambda)=\sum_{i=0}^p \lambda^i A_i$ and 
$X=[x_1,x_2,\dots, x_n]\in\mathcal{W}$.
Let $D_i=\Bdiag_{\tau_n}(W^{\HH}A_iW)$, $\mathcal{D}=\{D_i\}_{i=0}^p$,
$P_\mathcal{D}(\lambda)=\sum_{i=0}^p\lambda^i D_i$,
$\mu_j=\mbox{argmin}_{\mu\in\lambda(P_\mathcal{D})} g(\mu, \lambda_j; P_\mathcal{D}(\lambda))$,
$y_j$ be the eigenvector of $P_\mathcal{D}(\lambda)$ corresponding with $\mu_j$, for $j=1,2,\dots, n$.
Further assume that the geometric multiplicities of $\mu_j$'s all equal one,
$Y=[y_1,y_2,\dots, y_n]$ is nonsingular and $(\tau_n, I_n)$ is a solution to the exact GJBD problem of $\mathcal{D}$.
Denote
\begin{subequations}
\begin{align}
g_j&= g(\mu_j, \lambda_j; P_\mathcal{D}(\lambda)),\quad \mbox{for } j=1,2,\dots, n,\\
\eta &= \max_{1\le j\le n} \frac{ g_j\sum_{i=0}^p |\lambda_j|^i \|W^{\HH}A_iW\|_2}{\sqrt{\sum_{i=0}^{p-1} |\lambda_j|^{2i}}}.
\end{align}
\end{subequations}
Then 
there exists a permutation matrix $\Pi$ such that
\begin{align}\label{agjbd}
\|\OffBdiag_{\tau_n}(\Pi^{\T}X^{\HH} A_i X\Pi)\|_2\le C\mu\|X^{\HH} A_i X\|_2,
\end{align}
where
\[
C= \kappa_2(W^{-1}X)\left[\sqrt{n}\eta+\kappa_2(W^{-1}X)(1+\sqrt{n}\mu\eta)(1+\sqrt{n}\mu\eta+\sqrt{n}\eta)\right].
\]
\end{theorem}

\begin{proof}
First,  using the assumption that $(\tau_n,I_n)$ is a solution to the exact GJBD problem of $\mathcal{D}$, together with Theorem~\ref{thm1},
we know that there exists a permutation matrix $\Pi$ such that $Y\Pi\in\dtau$.

Second, for each $1\le j\le n$, take $(\lambda_j, W^{-1}x_j)$ as an approximate eigenpair of $P_\mathcal{D}(\lambda)$.
Using Lemma~\ref{lem:yuji}, we have
\begin{align}\label{sinwyx}
\sin\angle(W^{-1}x_j, y_j)\le \frac{g_j\|P_\mathcal{D}(\lambda_j)W^{-1}x_j\|_2}{\sqrt{\sum_{i=0}^{p-1} |\lambda_j|^{2i}}\|W^{-1}x_j\|_2}.
\end{align}
Let $\mathcal{E}=\{E_i\}_{i=0}^p$, $P_{\mathcal{E}}(\lambda)=\sum_{i=0}^p\lambda^i E_i$,
it holds that $P_\mathcal{D}(\lambda)+P_{\mathcal{E}}(\lambda)=W^{\HH}P_{\mathcal{A}}(\lambda)W$.
Then it follows from \eqref{offbwaw} and  \eqref{sinwyx} that
\begin{align*}
\sin\angle(W^{-1}x_j, y_j)
&\le \frac{g_j(\|W^{\HH}P_{\mathcal{A}}(\lambda_j)x_j\|_2+\|P_{\mathcal{E}}(\lambda_j) W^{-1}x_j\|_2)}
{\sqrt{\sum_{i=0}^{p-1} |\lambda_j|^{2i}}\|W^{-1}x_j\|_2}\\
&\le \frac{g_j\|P_{\mathcal{E}}(\lambda_j)\|_2 }{\sqrt{\sum_{i=0}^{p-1} |\lambda_j|^{2i}}}
\le \frac{g_j\sum_{i=0}^p |\lambda_j|^i \|E_i\|_2 }{\sqrt{\sum_{i=0}^{p-1} |\lambda_j|^{2i}}}\\
&\le \frac{\mu g_j\sum_{i=0}^p |\lambda_j|^i \|W^{\HH}A_iW\|_2}{\sqrt{\sum_{i=0}^{p-1} |\lambda_j|^{2i}}}\le \mu\eta.
\end{align*}
By Lemma~\ref{lem2}, there exists a $t_j\in\mathbb{C}$ such that
\[
\|W^{-1}x_j-t_jy_j\|_2=\|W^{-1}x_j\|_2 \sin\angle(W^{-1}x_j,y_j).
\]
Let $f_j=W^{-1}x_j-t_jy_j$, then it holds that $\|f_j\|_2\le \mu\eta \|W^{-1}x_j\|_2$.

Now denote $\widehat{Y}=[t_1y_1,\dots, t_ny_n]$, $F=[f_1,\dots, f_n]$,
we have $X=W(\widehat{Y}+F)$ and
\begin{equation}\label{fbound}
\|F\|_2\le\|F\|_F\le \mu\eta\|W^{-1}X\|_F\le \sqrt{n}\mu\eta\|W^{-1}X\|_2.
\end{equation}
Finally, direct calculations give rise to
\begin{subequations}
\begin{align}
&\mbox{}\quad\,\, \|\OffBdiag_{\tau_n}(\Pi^{\T}X^{\HH} A_i X\Pi)\|_2\notag\\
&=\|\OffBdiag_{\tau_n}(\Pi^{\T}X^{\HH}A_iX\Pi - \Pi^{\T}\widehat{Y}^{\HH} D_i\widehat{Y}\Pi)\|_2
\label{xax0}\\
&=\|\OffBdiag_{\tau_n}(\Pi^{\T}X^{\HH}A_iX\Pi- \Pi^{\T}\widehat{Y}^{\HH}W^{\HH}A_iW\widehat{Y}\Pi+\Pi^{\T} \widehat{Y}^{\HH}E_i\widehat{Y}\Pi)\|_2
\label{xax1}\\
&\le \|X^{\HH}A_iX- \widehat{Y}^{\HH}W^{\HH}A_iW\widehat{Y}+\widehat{Y}^{\HH}E_i\widehat{Y}\|_2\notag\\
&=\| \widehat{Y}^{\HH}W^{\HH}A_iWF + F^{\HH}W^{\HH}A_iX+\widehat{Y}^{\HH}E_i\widehat{Y}\|_2\label{xax2}\\
&\le \|\widehat{Y}\|_2 \|F\|_2\|W^{\HH}A_iW\|_2+ \|X^{-1}W\|_2 \|F\|_2 \|X^{\HH}A_iX\|_2+\|\widehat{Y}\|_2^2\|E_i\|_2\notag\\
&\le (\|W^{-1}X-F\|_2\|X^{-1}W\|_2^2  +\|X^{-1}W\|_2)\|F\|_2 \|X^{\HH}A_iX\|_2\notag\\
&\mbox{}\hspace{5.7cm}+\mu\|W^{\HH}A_iW\|_2\|W^{-1}X-F\|_2^2
\label{xax3}\\
&\le [1+\kappa_2(W^{-1}X)(1+\sqrt{n}\mu\eta)]\kappa_2(W^{-1}X)\sqrt{n}\mu\eta\|X^{\HH}A_iX\|_2\notag\\
&\mbox{}\hspace{4.5cm}+\mu(1+\sqrt{n}\mu\eta)^2\kappa_2^2(W^{-1}X)\|X^{\HH}A_iX\|_2
\notag\\
&=\mu C\|X^{\HH}A_iX\|_2,\notag
\end{align}
\end{subequations}
where
\eqref{xax0} uses $\OffBdiag_{\tau_n}(\Pi^{\HH}\widehat{Y}^{\HH} D_i\widehat{Y}\Pi)=0$
since $\widehat{Y}\Pi\in\dtau$,
\eqref{xax1} uses $D_i=W^{\HH}A_iW-E_i$,
\eqref{xax2} uses $X=W(\widehat{Y}+F)$,
\eqref{xax3} uses $\widehat{Y}=W^{-1}X-F$ and \eqref{offbwaw}.
This completes the proof.

\end{proof}

Several remarks follow.
\begin{remark}\label{rem1}
The assumption that ``$(\tau_n, I_n)$ is a solution to the exact GJBD problem of $\mathcal{D}$ ''  is equivalent to say that ``$\mathcal{D}$ can not be further block diagonalized''.
\end{remark}

\begin{remark}\label{rem2}
If $\mu=0$, the approximate GJBD problem becomes the exact GJBD problem. The conclusion in Theorem~\ref{thm3} agrees with that in Theorem~\ref{thm1}.
\end{remark}

\begin{remark}\label{rem3}
The inequality \eqref{agjbd} implies that $X\Pi$ is a ``sub-optimal'' diagonalizer.
The constant $C$ plays a crucial role in bounding the off-block diagonal part of $\Pi^{\T}X^{\HH}A_iX\Pi$.
Notice that the condition number of $W^{-1}X$ dominates the value of $C$. 
When $W$ and $X$ are good conditioned, $\kappa_2(W^{-1}X)$ is small.
However, if $W$ or $X$ is ill-conditioned, $\kappa_2(W^{-1}X)$ can be quite large,
which means that $X\Pi$ can be ``low-quality''.
In fact, from the perturbation theory of JBD problem \cite{cai2017perturbation}
 (see also perturbation theory of JD problem in \cite{shi2015some}, \cite{afsari2008sensitivity}),
the diagonalizer is sensitive to the perturbation when it is ill-conditioned.
Therefore, it is not surprising to draw the conclusion that $C$ can be large when $W$ or $X$ is ill-conditioned.

\end{remark}

\section{Numerical Method}\label{sec:method}
According to Theorems~\ref{thm1} and \ref{thm3},
solutions to the exact/approximate GJBD problem can be obtained
by the following three-stage procedure:
\begin{description}
\item[Stage 1 -- eigenproblem solving stage.] Compute $n$ linearly independent eigenvectors $x_1,\ldots,x_n$ 
of the matrix polynomial $P_{\mathcal{A}}(\lambda)$, and let ${X}=[x_1, \dots, x_n]$;

\item[Stage 2 -- block structure revealing stage.] Determine $\tau_n$ and a permutation matrix $\Pi$ such that $\Pi^{\T}{X}^{\HH} A_i{X}\Pi$  for $i=0,1,\ldots,p$ are all approximately $\tau_n$-block diagonal
and $\card(\tau_n)$ is maximized.

\item [Stage 3 -- refinement stage.] Refine $X\Pi$ to improve the quality of the diagonalizer.
\end{description}
We call the above three-stage method
{\em a partial eigenvector approach with refinement (PEAR) method} for the GJBD problem.
Next, we will discuss the implement details of each stage of the above procedure.

\subsection{Stage 1.}
In this stage, we compute $n$ linearly independent eigenvectors of PEP
$P_{\mathcal{A}}(\lambda)$.
By \cite[Section 3]{pereira2003solvents}, the number of choices for $n$ linearly independent eigenvectors is no less than $p$.
How shall we make such a choice?
To answer this question, let us consider the following GJBD problem first.
\begin{example}\label{eg1}
Let $A_i=VD_iV^{\HH}$ for $ i=0, 1, 2$,
where
\begin{align*}
V&=\begin{bmatrix}
0.5377 + 2.7694\imath & 0.8622 + 0.7254\imath & -0.4336 - 0.2050\imath \\
1.8339 - 1.3499\imath & 0.3188 - 0.0631\imath & 0.3426 - 0.1241\imath \\
-2.2588 + 3.0349\imath & -1.3077 + 0.7147\imath & 3.5784 + 1.4897\imath \\
\end{bmatrix},\\
D_0&=
\begin{bmatrix}
0.2939 - 0.7873\imath & 0.0137 + 0.0044\imath & -0.0171 + 0.0038\imath \\
0.0032 - 0.0086\imath & 0.8884 - 2.9443\imath & -1.0689 + 0.3252\imath \\
0.0031 - 0.0003\imath & -1.1471 + 1.4384\imath & -0.8095 - 0.7549\imath \\
\end{bmatrix},\\
D_1&=
\begin{bmatrix}
-0.1649 + 0.6277\imath & -0.0077 + 0.0022\imath & 0.0037 + 0.0075\imath \\
-0.0109 + 0.0055\imath & 1.0933 - 1.2141\imath & -0.8637 - 0.0068\imath \\
0.0003 + 0.0110\imath & 1.1093 - 1.1135\imath & 0.0774 + 1.5326\imath \\
\end{bmatrix},\\
D_2&=
\begin{bmatrix}
1.5442 + 0.0859\imath & -0.0076 + 0.0025\imath & -0.0140 + 0.0062\imath \\
-0.0018 + 0.0142\imath & -1.4916 - 0.6156\imath & -1.0616 - 0.1924\imath \\
-0.0020 + 0.0029\imath & -0.7423 + 0.7481\imath & 2.3505 + 0.8886\imath \\
\end{bmatrix}.
\end{align*}
Let $W=V^{-\HH}$, then we know that $(\tau_n, W)$ is a solution to the approximate GJBD problem of $\{A_0,A_1,A_2\}$, where $\tau_n=(1,2)$.

By calculations, we know that all eigenvalues of the PEP $P_{\mathcal{A}}(\lambda)=\lambda^2 A_2+\lambda A_1 +A_0$
are simple, and the eigenvector matrix $X$, whose columns are all of unit length, can be given by
\begin{align*}
X(:,1:3)&=\begin{bmatrix}
-0.5269 - 0.3384\imath & -0.3566 + 0.4066\imath & 0.0646 - 0.5182\imath  \\
-0.1210 - 0.7587\imath & -0.8248 + 0.0209\imath & 0.0578 - 0.7205\imath  \\
-0.1247 + 0.0436\imath & -0.0681 + 0.1491\imath & -0.4029 - 0.2062\imath
\end{bmatrix},
\\
X(:,4:6)&=
\begin{bmatrix}
0.1636 - 0.3165\imath & -0.2841 + 0.2244\imath & 0.4363 - 0.4082\imath \\
0.7059 + 0.6000\imath & -0.4017 - 0.8314\imath & 0.7443 - 0.1520\imath \\
0.0065 - 0.1214\imath & -0.0599 + 0.1128\imath & -0.0381 - 0.2541\imath
\end{bmatrix}.
\end{align*}
Let  $\widehat{W}$ be a 3-by-3 matrix whose columns are selected from the columns of $X$.
Define
\begin{align*}
f(\tau_n,\widehat{W})&=\sum_{i=0}^2 \|\OffBdiag_{\tau_n}(\widehat{W}^{\HH}A_i\widehat{W})\|_F^2,\\
\theta(W,\widehat{W})&=\max\{ \angle(W(:,1),\widehat{W}(:,1)),
\angle(W(:,[2, 3]),\widehat{W}(:,[2, 3]))
\}.
\end{align*}
For four different choices of $\widehat{W}$,
we compute the condition number of $\widehat{W}$, $f(\tau_n,\widehat{W})$ and $\theta(W,\widehat{W})$, and the results are listed in Table~\ref{tab1}.
\begin{table}[ht]
\begin{center}
\begin{tabular}{c||c|c|c|c}\toprule
Case  & $\widehat{W}$ & $\mbox{cond}(\widehat{W})$ & $f(\tau_n,\widehat{W})$  & $\theta(W,\widehat{W})$ \\ \hline\hline
1 & $X(:,[4, 1, 2])$ & 2.3e1 &0.0048 & 0.0066 \\ \hline
2 & $X(:,[4, 2, 3])$ & 9.2e0 &0.1061 & 0.0154 \\ \hline
3 & $X(:,[6, 4, 5])$ & 4.9e2 &0.0247 & 0.8855 \\ \hline
4 & $X(:,[1, 2, 3])$ & 1.2e3 &9.5550 & 0.5005
  \\ \bottomrule
\end{tabular}
\caption{Condition number, cost function $f(\tau_n,\widehat{W})$ and angle $\theta(W,\widehat{W})$}
\end{center}
\label{tab1}
\end{table}
\end{example}

From Example~\ref{eg1} we can see that
the qualities of the four approximate diagonalizers $\widehat{W}$
are quite different although they are all consisted of linearly independent eigenvectors:
$\widehat{W}$ in Cases 1 and 2 are good in the sense that
$f(\tau_n,\widehat{W})$ and $\theta(W,\widehat{W})$ are  small,
and $\widehat{W}$ in Case 1 is better;
$\widehat{W}$ in Case 3 is not good though $f(\tau_n,\widehat{W})$ is small;
$\widehat{W}$ in Case 4 is even not an approximate diagonalizer.

Based on the above observations and also Remark~\ref{rem3},
it is reasonable to choose $n$ linearly independent eigenvectors $x_1, x_2, \dots, x_n$ such that
the condition number of $[x_1, x_2, \dots, x_n]$ is as small as possible.
Then the task of Stage 1 is reduced to find such $n$ linearly independent eigenvectors.
Classic eigensolvers for PEP concentrate on computing
extreme eigenvalues or eigenvalues close to a prescribed number
(and their corresponding eigenvectors),
which are not suitable for the task.
In this paper, we use  the following two steps to accomplish the task:

\medskip

\noindent{Step 1.} Compute $k$ unit length eigenvectors $x_1, \dots, x_{k}$ of the PEP by certain eigensolvers, where $k$ is larger than $n$, say a multiple of $n$, $k=2n, 4n$.
In our numerical tests, we first transform the PEP into the GEP $(\lambda M+N)u=0$,
where $u$, $M$ and $N$ are given by \eqref{uxl}, \eqref{mn}, respectively.
Then when $p$ is small, we use QZ method to find all eigenvectors of the GEP;
When $p$ is large, we use Arnoldi method to compute $k$ largest magnitude eigenvalues and the corresponding eigenvectors of the GEP.
Finally, the eigenvectors of the PEP can be obtained via those of the GEP.

\medskip

\noindent{Step 2.} Compute the QR decomposition of $[x_1,  \dots, x_k]$ with column pivoting,
i.e., $[x_1,\dots, x_k]P=QR$, where $P$ is a permutation matrix of order $k$, $Q$ is unitary, $R$ is upper triangular with main diagonal entries in a decreasing order.
If the $(n,n)$ entry of $R$ is small,
\footnote{This indicates that  $x_1,\dots, x_k$ are almost linearly dependent,
and it is generally impossible to find a good diagonalizer via those eigenvectors.}
return to Step 1 to find more eigenvectors,
else set $X$ as the first $n$ columns of $[x_1,\dots, x_k]P$.

The above two-step procedure is perhaps the simplest way to accomplish the task of Stage 1,
but it maybe still worth developing some particular eigensolvers for it.

\subsection{Stage 2.}
In this stage, we need to determine a partition $\tau_n$ and a permutation matrix $\Pi$
such that $\Pi^{\T}X^{\HH}A_i X\Pi$'s are all approximately $\tau_n$-block diagonal
and $\card(\tau_n)$ is maximized.
This stage is of great importance for determining the solutions to the GJBD problem,
but without knowing the number of the diagonal blocks,
determining a correct $\tau_n$ can be very tricky,
especially when the noise is high and the block diagonal structure is fussy.
From our numerical experience, the spectral clustering method \cite{von2007tutorial} is powerful and efficient for finding $\tau_n$ and $\Pi$.
Let
\begin{align}
H=[h_{ij}]=\sum_{i=0}^p(|X^{\HH}A_iX|+|X^{\HH}A_i^{\HH}X|),
\end{align}
we can define some (normalized) graph Laplacian matrix $L$ from $H$.
For example, in this paper, we set
\begin{align}\label{eq:l}
L=[l_{ij}] \quad \mbox{with} \quad
l_{ij}=
\left\{
  \begin{array}{cc}
 -1, & \mbox{if} \ i\ne j, \ h_{ij}> \sum_{k=1}^n h_{kj}/n; \\
 0, & \mbox{if} \ i\ne j, \ h_{ij}\le \sum_{k=1}^n h_{kj}/n; \\
 -\sum_{k\ne j} l_{kj}, & \mbox{if}\ i=j.
  \end{array}
\right.
\end{align}
Then the number of diagonal blocks equal the multiplicity of zero as an eigenvalue of $L$,
using $k$-means method \cite{arthur2007k}\cite{macqueen1967some},
$\tau_n$ and $\Pi$ can be determined by clustering the eigenvectors corresponding to eigenvalue zero.


\subsection{Stage 3.}
Theorem~\ref{thm3} only ensures that $\widehat{W}=X\Pi$ is a ``sub-optimal'' diagonalizer.
In order to improve the quality of the diagonalizer, we propose the following refinement procedure.

Let $(\tau_n, \widehat{W})$ be an approximate solution to the GJBD problem produced by the first two stages of PEAR.
Suppose $\tau_n=(n_1,\dots,n_t)$,
$\widehat{W}
=[\widehat{W}_1, \dots, \widehat{W}_t]$
with $\widehat{W}_j\in\mathbb{C}^{n\times n_j}$ for $j=1,\dots,t$.
Denote $\widehat{W}_{-j}=[W_1, \dots, W_{j-1}, W_{j+1}, \dots, W_t]$
and
\begin{align}
\mathscr{B}_j=\left\{
\begin{array}{cc}
\left[A_0 \widehat{W}_{-j},
  A_1\widehat{W}_{-j},
\dots,
 A_p \widehat{W}_{-j}\right]^{\HH}, & \mbox{if $A_j$'s are Hermitian;} \\
\left[A_0 \widehat{W}_{-j}, A_0^{\HH}\widehat{W}_{-j},
\dots,
A_p \widehat{W}_{-j}, A_p^{\HH}\widehat{W}_{-j}\right]^{\HH}
& \mbox{Otherwise.}
\end{array}
\right.
\end{align}
Fixing $\widehat{W}_{-j}$, we can minimize
$f(\tau_n,\widehat{W})=\sum_{i=0}^p \|\OffBdiag_{\tau_n}(\widehat{W}^{\HH}A_i\widehat{W})\|_F^2$
\footnote{$f(\tau_n,\widehat{W})$ is in fact a commonly used cost function for JBD problem.
One can of course use some other cost functions,
but this one seems the most simple one for our refinement purpose.}
by updating $\widehat{W}_j$ as $V_j$,
where the column vectors of $V_j$ are the $n_j$ right singular vectors of $\mathscr{B}_j$
corresponding with the $n_j$ smallest singular values.
For $j=1,2,\dots, t$, we update $\widehat{W}_j$ as above, we call it a refinement loop.
We can repeat the refinement loop until the diagonalizer is sufficiently good.
In our numerical test, three refinement loops are sufficient.

Note that the effectiveness of this refinement procedure is built on
the assumption that $\tau_n$ obtained in Stage 2 is correct.
Without such an assumption, the refinement procedure may make the diagonalizer even more worse.

It is also worth mentioning here that the above refinement procedure can be used to update
any approximate diagonalizer, not necessarily the one produced by Stage 2.
As a matter of fact, the procedure itself can be used to find a diagonalizer,
but without a good initial guess, the convergence can be quite slow.

\begin{remark}
In Stage 1, assuming each eigenpair can be found in $\mathcal{O}(1)$ steps,
then $k=\mathcal{O}(n)$ eigenpairs can be obtained in $\mathcal{O}(n^3)$ steps.
In Stage 2, first a symmetric eigenvalue problem needs to be solved, which can be done in $\mathcal{O}(n^3)$ flops; second, $k$-means needs to be performed,
which is so fast in practice that we can ignore its cost.
Stage 3 requires $\mathcal{O}(n^3)$ flops, including the matrix-matrix multiplications and singular value decomposition (SVD), etc.
So the overall computational cost of PEAR is $\mathcal{O}(n^3)$ flops.
\end{remark}

When the matrices $A_i$'s are all real, a real diagonalizer is required.
But the diagonalizer returned by PEAR is in general complex.
Can we make some simple modifications to PEAR to get a real diagonalizer?
The answer is positive. 
In fact, we can add the following ``Stage 2.5'' to get a real diagonalizer from PEAR.

\subsection{Stage 2.5}
Let $\widehat{W}=[\widehat{W}_1, \dots, \widehat{W}_t]$ be the approximate diagonalizer at the beginning of Stage 3,
where $\widehat{W}_j\in\mathbb{C}^{n\times n_j}$ for $j=1,\dots,t$.
For each $j$, denote $\widehat{W}_j=\widehat{W}_{jR}+\imath \widehat{W}_{jI}$, 
where $\widehat{W}_{jR}, \widehat{W}_{jI}\in\mathbb{R}^{n\times n_j}$ are the real and imaginary parts of $\widehat{W}_j$, respectively .
Let the SVD of $[\widehat{W}_{jR}, \widehat{W}_{jI}]$ be $[\widehat{W}_{jR}, \widehat{W}_{jI}]=U_j\Sigma_j V_j^{\T}$,
where $U_j\in\mathbb{R}^{n\times n}, V_j\in\mathbb{R}^{2n_j\times 2n_j}$ are orthogonal, 
$\Sigma_j\in\mathbb{R}^{n\times 2n_j}$ is diagonal with its main diagonal nonnegative and in a decreasing order.
Then we update $\widehat{W}_j$ as $\widehat{W}_j=U_j(:,1:n_j)$.
As a consequence, $\widehat{W}=[\widehat{W}_1, \dots, \widehat{W}_t]$ will be a real diagonalizer.

The mechanic behind the above procedure is the following critical assumption:

A.) {\em All solutions to the GJBD problem of $\mathcal{A}$ are equivalent, 
and the intersection of all diagonalizers and $\mathbb{R}^{n\times n}$ is nonempty.}

Based on the above assumption, for the diagonalizer $\widehat{W}\in\mathbb{C}^{n\times n}$ at the beginning of Stage 3,
there exists a nonsingular $\tau_n$-block diagonal matrix $D=\diag(D_{11},\dots, D_{tt})$ such that 
$\widetilde{W}:=\widehat{W}D^{-1}$ is a real diagonalizer.
Partition $\widetilde{W}$ as $=[\widetilde{W}_1, \dots, \widetilde{W}_t]$ with $\widetilde{W}_j\in\mathbb{C}^{n\times n_j}$,
and let $D_{jj}=D_{jR}+\imath D_{jI}$, where $D_{jR}, D_{jI}\in\mathbb{R}^{n_j\times n_j}$. 
Then by $\widetilde{W}_j=\widehat{W}_jD_{jj}^{-1}$, we have
\begin{align*}
\widetilde{W}_j [D_{jR}, D_{jI}] =[\widehat{W}_{jR},\widehat{W}_{jI}].
\end{align*}
Then it follows that $\subspan([\widehat{W}_{jR},\widehat{W}_{jI}])=\subspan(\widetilde{W}_j [D_{jR}, D_{jI}])=\subspan(\widetilde{W}_j)$
since $[D_{jR}, D_{jI}]$ is of full row rank.
Therefore, updating $\widehat{W}_j$ as $U_j(:,1:n_j)$, we know that $\widehat{W}=[U_1(:,1:n_1),\dots, U_t(:,1:n_t)]$ is a real diagonalizer.

The  assumption A.) in general holds, but not always.
For example, consider the GJBD problem of $\{A_0,A_1\}$, where
\begin{align*}
A_0=\begin{bmatrix}1 & 1 & 1\\ 1 & 1& -3\\ -3 & 1& 1\end{bmatrix},\quad
A_1=\begin{bmatrix}3 & -1 & 1\\ -1 & 3& -3\\ -3 & 1& 3\end{bmatrix}.
\end{align*}
Then by calculations, we know that $((1, 2), \bsmat 1& 1 & 0\\ 1& 0 & 1\\ 0& 1& 1\esmat)$ and $((1, 1, 1), \bsmat 1 & 1-\imath & 1+\imath \\ 1 & 1+\imath & 1-\imath \\ 0 & 2 & 2\esmat)$
are two inequivalent solutions to the GJBD problem with the diagonalizers in $\mathbb{R}^{n\times n}$ and $\mathbb{C}^{n\times n}$, respectively. 
How to find a real diagonalizer when assumption A.) does not hold is difficult and needs further investigations.
In our numerical tests (subsection~\ref{eg2}), Stage 2.5 works perfectly for finding a real diagonalizer.

\section{Numerical Examples}\label{sec:eg}
In this section, we present several examples to illustrate the performance of PEAR.
All the numerical examples were carried out on a quad-core $\text{Intel}^{\tiny\textregistered} \text{Xeon}^{\tiny\textregistered}$ Processor E5-2643 running at 3.30GHz with 31.3GB AM,
using MATLAB R2014b with machine $\epsilon=2.2\times 10^{-16}$. 

We compare the performance of PEAR (with and without refinement) 
with the second GJBD algorithm in \cite{cai2017algebraic}, namely, $\star$-commuting based method with a conservative strategy, SCMC for short,
and also two algorithms for the JBD problem, namely, JBD-LM \cite{cherrak2013non} and JBD-NCG \cite{nion2011tensor}.
For PEAR, three refinement loops are used to improve the quality of the diagonalizer.
For SCMC, the tolerance is set as $3n^210^{-\mbox{SNR}/20}$,
where $\mbox{SNR}$ is the signal-to-noise ratio defined below.
For the JBD-LM method, the stopping criteria are $\|W_{k+1}-W_k\|_F<10^{-12}$, 
or $|\frac{\phi_k-\phi_{k+1}}{\phi_k}|<10^{-8}$ for successive 5 steps, or the maximum number of iterations,
which is set as 200, exceeded.
For the JBD-NCG method, the stopping criteria are $|\phi_k-\phi_{k+1}|<10^{-8}$, or $|\frac{\phi_k-\phi_{k+1}}{\phi_k}|<10^{-8}$ for successive 5 steps, 
or the maximum number of iterations, which is set as 2000, exceeded.
Here $W_k$, $\phi_k$ are the $W$ matrix and the value of the cost function in $k$th step, respectively.
In JBD-LM and JBD-NCG algorithms, 20 initial values (19 random initial values and an EVD-based initial value \cite{nion2011tensor}) are used to iterate 20 steps first, and then the iteration which produces the smallest value of the cost function proceeds until one of the stopping criteria is satisfied.

\subsection{Random data}\label{eg1}

Let $\tau_n=(n_1,n_2,\dots,n_t)$ be a partition of $n$,
we will generate the matrix set $\mathcal{A}=\{A_i\}_{i=0}^p$ by the following model:
 \begin{align}\label{randmodel}
   A_i=V^{\HH} D_iV, \quad i=0,1,\ldots,p,
 \end{align}
where $V$ and $D_i$ are, respectively, the mixing matrix and the approximate $\tau_n$-block diagonal matrices.
The elements in V and $\Bdiag_{\tau_n}(D_i)$ are all complex numbers whose real and imaginary parts are drawn
from a standard normal distribution, while the elements in $\OffBdiag_{\tau_n}(D_i)$ are all complex numbers whose real and imaginary parts are drawn from a normal distribution with mean zero and variance $\sigma^2$. The signal-to-noise ratio is defined as $\mbox{SNR}=10\log(1/\sigma^2)$.

For model \eqref{randmodel}, we define the following performance index
to measure the quality of the computed diagonalizer $W$, which is used in \cite{cai2017algebraic} and \cite{cai2015matrix}:
\begin{align*}
\PI(V^{-1},W)=\min_{\pi}\max_{1\leq i\leq t}\subspace(V_i, W_{\pi(i)}),
 \end{align*}
where $V^{-1}=[V_1, V_2,\dots, V_t]$,
$W=[W_1, W_2,\dots, W_t]$, $V_i, W_{\pi(i)}\in\mathbb{C}^{n\times n_i}$ for $i=1,2,\ldots, t$, the vector $(\pi(1),\pi(2),\ldots,\pi(t))$ is a permutation of $\{1,2,\dots t\}$ satisfying $(n_{\pi(1)},n_{\pi(2)},\ldots,n_{\pi(t)})=\tau_n$,
the expression $\subspace(E,F)$ denotes the angle between two subspaces specified by the columns of $E$ and $F$,
which can be computed by MATLAB function ``{subspace}''.

In what follows,  we generate the matrix set by model \eqref{randmodel} with the following parameters:
\begin{description}
\item[P1.] $n=9$, $\tau_n=(3,3,3)$, $p+1=25$.
\item[P2.] $n=9$, $\tau_n=(2,3,4)$, $p+1=25$.
\item[P3.] $n=9$, $\tau_n=(2,3,4)$, $p+1=20,40,\dots,200$, SNR=80.
\item[P4.] $n=9m$, $\tau_n=(2m,3m,4m)$ for $m=1,2,\dots,6$, $p+1=10$, SNR=80.  
\end{description}

\textbf{Experiment 1.}
For different SNRs, we generate the data with parameters P1 and P2, respectively.
Then for each matrix set generated by those parameters, we perform PEAR for 1000 independent runs.
It is known that when the SNR is small, PEAR may fail, 
specifically, the block diagonal structure in Stage 2 of PEAR is fuzzy, the computed $\hat{\tau}_n$ may not be consistent with the true $\tau_n$,
namely, there is no $(0,1)$ matrix $N$ such that $\tau_n=\hat{\tau}_nN$.
In Table~\ref{tab:percentage}, we list the percentages of successful runs of {PEAR}.
From the table we can see that the smaller the SNR is, the more likely PEAR may fail.
In our tests, when SNR is no less than 60, PEAR  didn't fail.
\begin{table}[h]
\footnotesize
\centering
\begin{tabular}{c|cccccccc}\toprule
  \mbox{SNR} & $30$ & $40$ & $50$ & $60$ & $70$ & $80$ & $90$ & $100$  \\\hline
 P1.   & $74.1\%$ & $96.1\%$ & $99.7\%$ & $100\%$ & $100\%$ & $100\%$ & $100\%$& $100\%$  \\\hline
 P2.  & $65.9\%$ & $96.4\%$ & $99.6\%$ & $100\%$ & $100\%$& $100\%$ & $100\%$& $100\%$  \\\hline
\end{tabular}
\caption{The percentages of successful runs of {PEAR} over $1000$ independent runs}
\label{tab:percentage}
\end{table}


\textbf{Experiment 2.} For SNR= 40, 60, 80, 100,
we generate the data with parameters P1 and P2, respectively.
For each matrix set generated by those parameters, 
we perform JBD-LM, JBD-NCG, SCMC, PEA (PEAR without refinement) and PEAR for 50 independent runs,
then compare their performance indices.
The box plot (generated by MATLAB function ``boxplot'') of the results are displayed in Figures \ref{Fig_333_SNR} and \ref{Fig_234_SNR}.
\begin{figure}[!ht]
\centering
\includegraphics[width=1.00\textwidth]{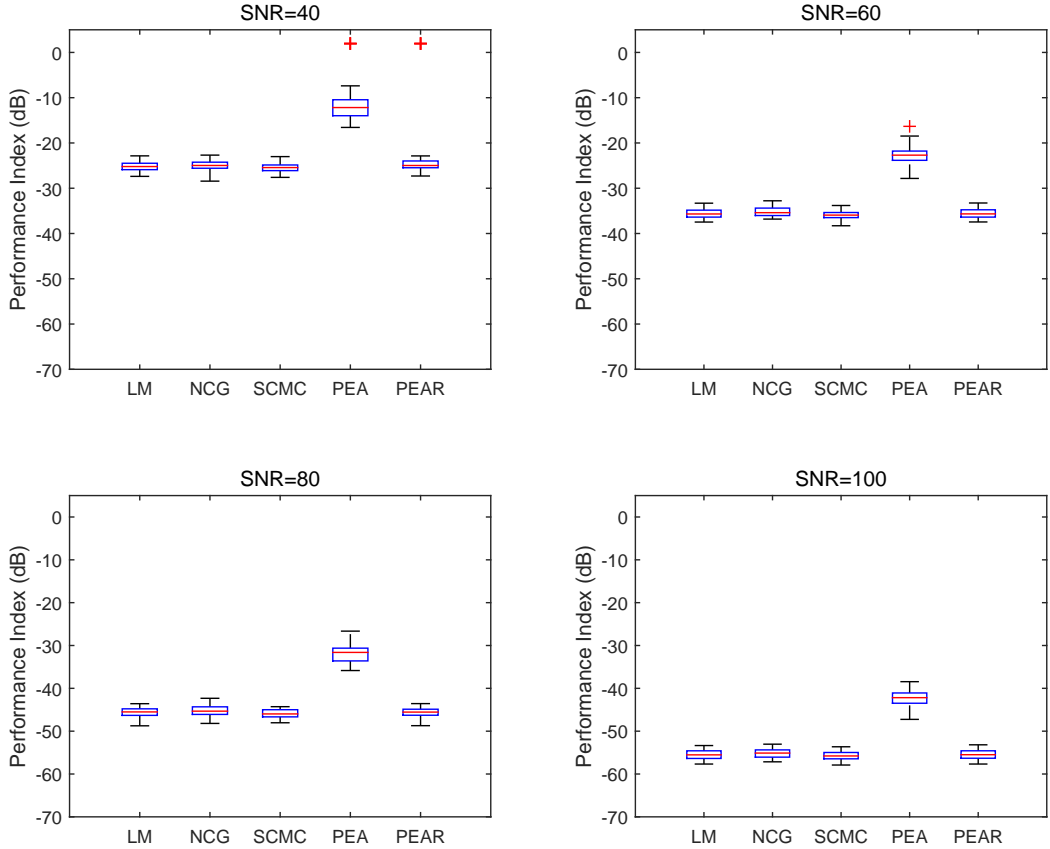}
\caption{Performance indices of five methods with different SNRs for P1}\label{Fig_333_SNR}
\end{figure}

\begin{figure}[!ht]
\centering
\includegraphics[width=1.00\textwidth]{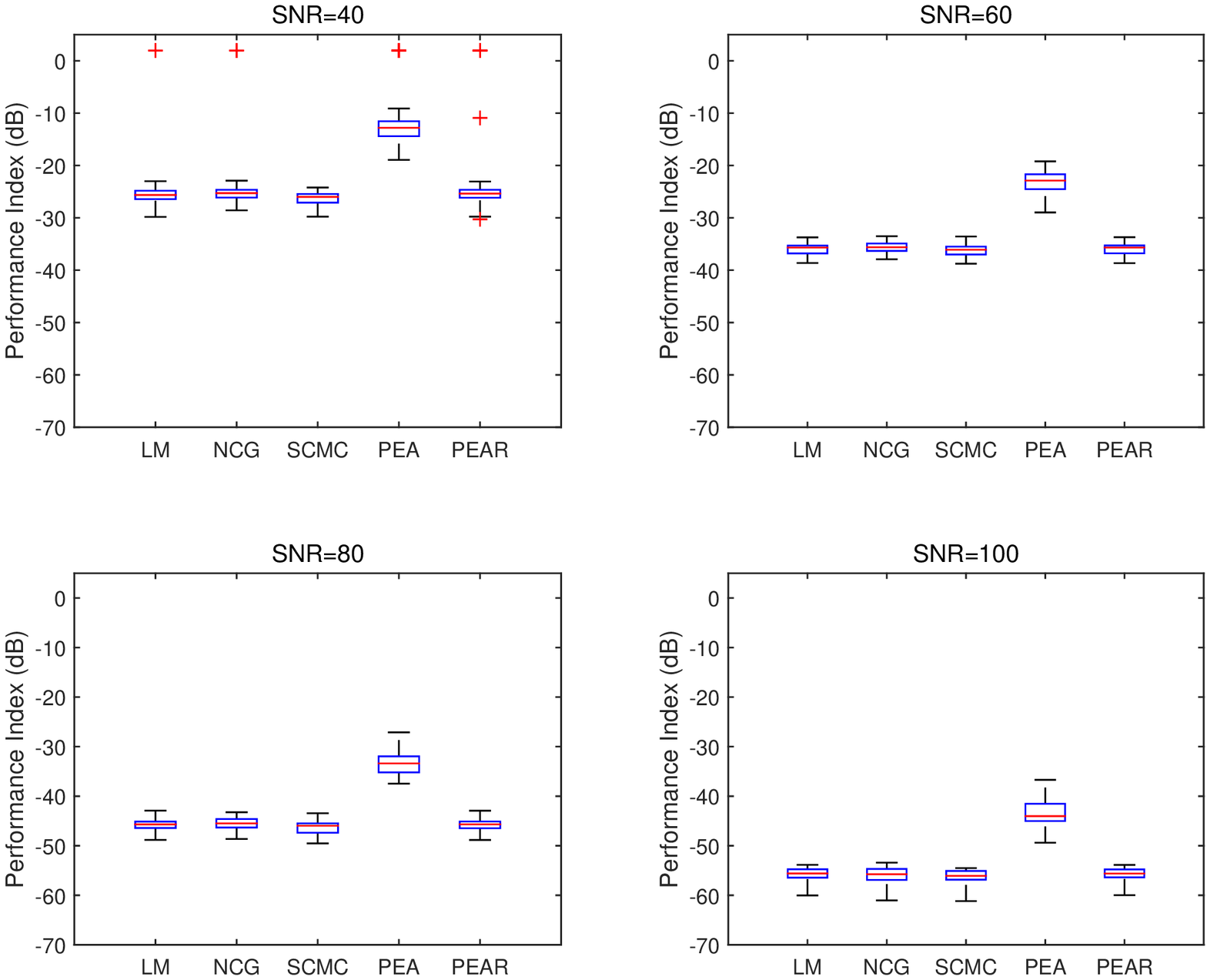}
\caption{Performance indices of five methods with different SNRs for P2}\label{Fig_234_SNR}
\end{figure}

We can see from Figure \ref{Fig_333_SNR} and Figure \ref{Fig_234_SNR}
that for both cases, when SNR equals  40, 60, 80 or 100,
the performance indices produced by JBD-LM, JBD-NCG, SCMC and PEAR are almost the same;
the performance indices produced by PEA are larger than those of other four methods,
which indicates that the diagonalizers produced by the first two stages of PEAR indeed suffer from low  quality,
and the refinement stage of PEAR is effective.
For all methods, the performance indices decrease as the SNR increases.

\textbf{Experiment 3.}
We generate the data with the parameters in P3 and P4, respectively. 
50 independent trials are performed for each matrix set.
The average CPU time of two algebraic methods -- SCMC and PEAR 
are displayed in Figures \ref{Fig_333_Num}.
From the left figure we can see that the CPU time of the two methods increase almost linearly with increased matrix number.
And from the right figure we can see that the CPU time of SCMC increases dramatically as matrix size increases,
meanwhile, the CPU time of PEAR  increases much slower.

\begin{figure}[!ht]
\centering
\includegraphics[width=0.49\textwidth]{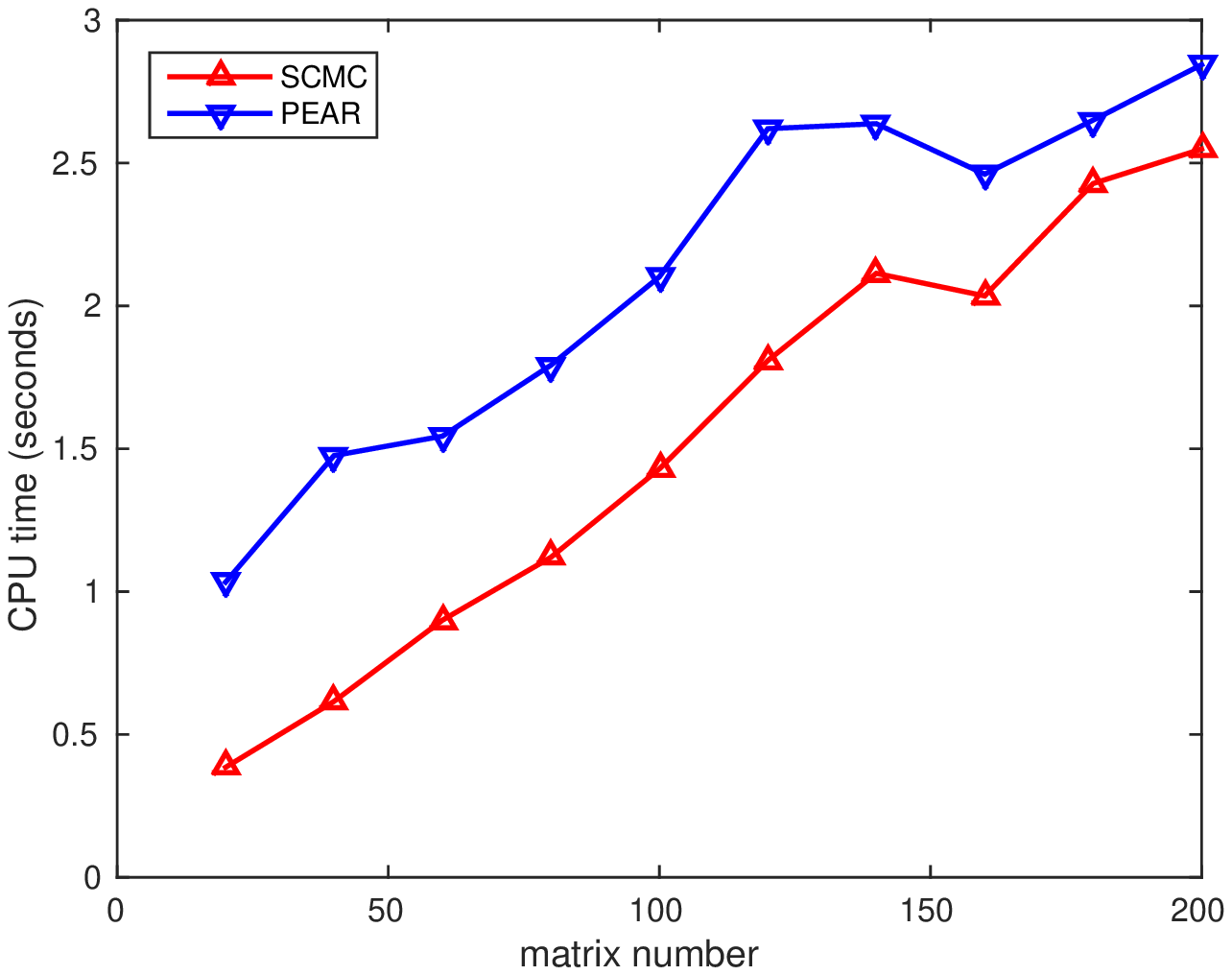}
\includegraphics[width=0.49\textwidth]{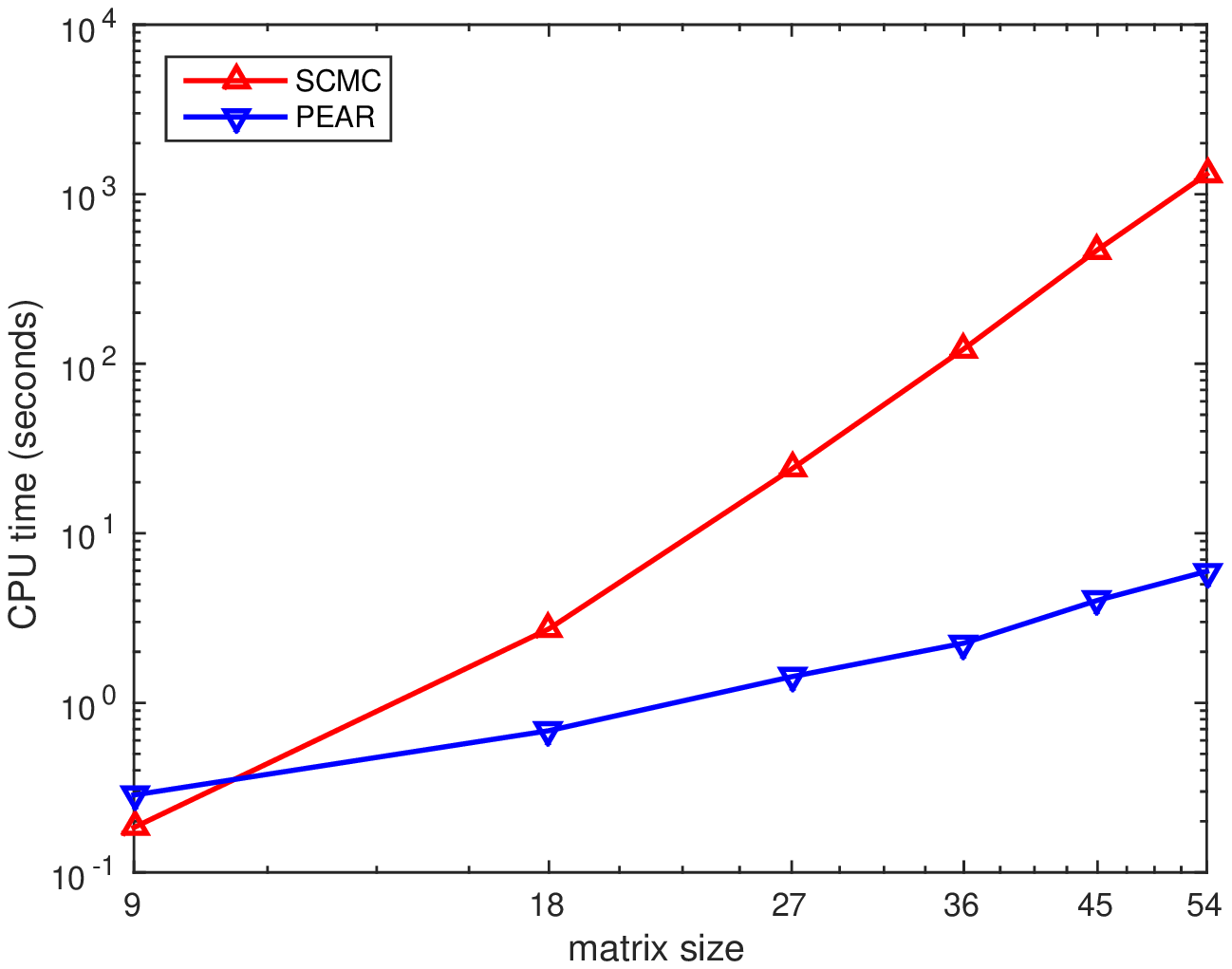}
\caption{(left): CPU time of two methods with different matrix numbers, 
(right): CPU time of two methods with different matrix sizes}\label{Fig_333_Num}
\end{figure}

%

%

\subsection{Separation of convolutive mixtures of source}\label{eg2}
We consider example 4.2 in \cite{cai2017algebraic}, where a real diagonalizer is required. 
All settings are kept the same.
In Figure \ref{Fig_cor}, for different
SNRs, we plot the correlations between the source signals and the extracted signals
obtained from computed solutions by JBD-LM, JBD-NCG, SCMC and PEAR, respectively. 
All displayed results have been averaged over 50 independent trials.

\begin{figure}[!ht]
  \centering
  \includegraphics[width=0.49\textwidth]{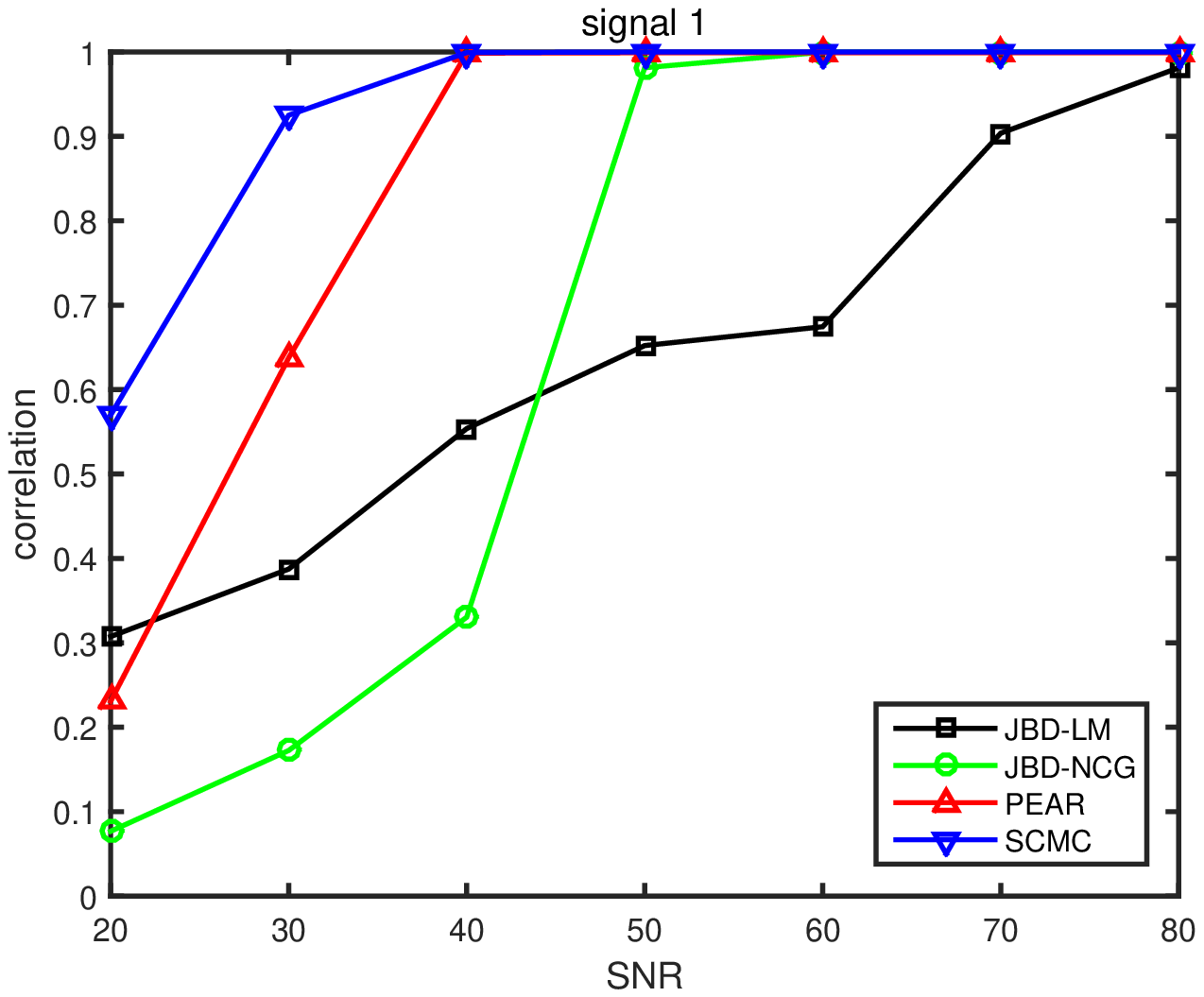}
  \includegraphics[width=0.49\textwidth]{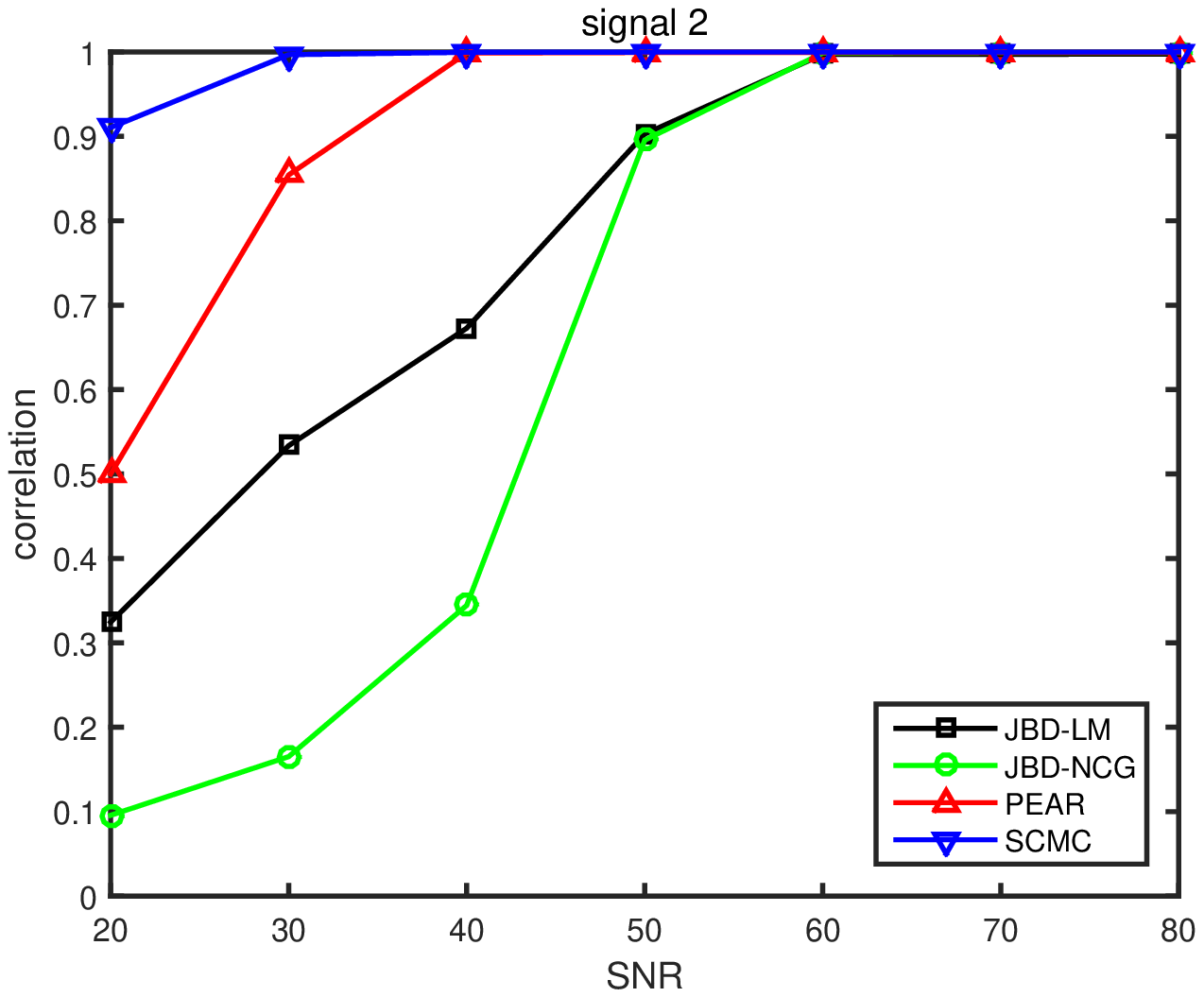}\\
  \includegraphics[width=0.49\textwidth]{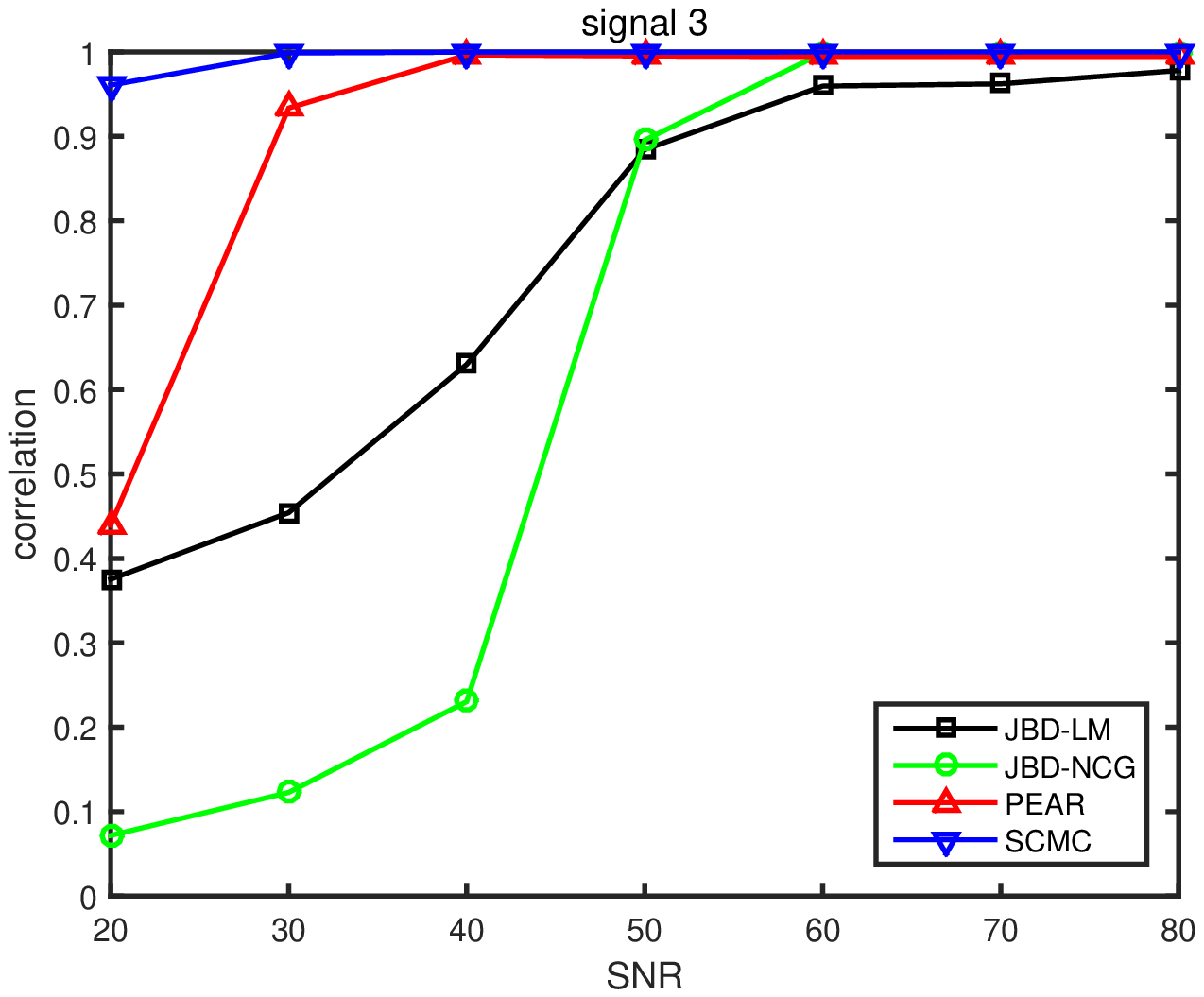}
  \caption{\text{Correlation between recovered signals and source signals.}}\label{Fig_cor}
  \end{figure}

We can see from Figure~\ref{Fig_cor} that when SNR is larger than 60, 
the recovered signals obtained from all four methods are all good approximations of the source signals;
when SNR is less than 60, SCMC is the best, PEAR is the second best.
The reason why SCMC is better than PEAR in this example is that
PEAR fails to find the correct partition in Stage 2 when SNR is small,
meanwhile, with a large tolerance for SCMC, SCMC is able to find a consistent partition with the correct one.
In Stage 2 of PEAR, if we use some normalized Laplacian to find the partition, 
the numerical results of PEAR can be improved.

\section{Conclusion}\label{sec:conclusion}

In this paper, we show how the GJBD problem of a matrix set is related to a matrix polynomial eigenvalue problem.
Theoretically, under mild conditions, we show that 
(a) the GJBD problem of $\{A_i\}_{i=0}^p$ can be solved by $n$ linearly independent eigenvectors of 
the matrix polynomial $P(\lambda)=\sum_{i=0}^p \lambda^i A_i$;
(b) all solutions to the GJBD problem are equivalent;
(c) a sub-optimal solution of the approximate GJBD problem can also be given by $n$ linearly independent eigenvectors.
Algorithmically, we proposed a three-stage method -- PEAR to solve the GJBD problem.
Numerical experiments show that PEAR is effective and efficient. 

Finally, it is worth mentioning here that the GJBD problem discussed in the paper, compared with the BTD of tensors,
is limited in several aspects \cite{cai2017algebraic}: the matrices are square rather than general nonsquare ones; 
the matrices are factorized via a congruence transformation rather than a general one, etc.
Is it possible to use the matrix polynomial approach in this paper to compute a blind BTD (BTD without knowing the number of terms and the size of each term) of tensors?
We will try to answer this question in our further work.


\begin{thebibliography}{10}

\bibitem{abed2004algorithms}
Abed-Meraim, K., Belouchrani, A.: Algorithms for joint block diagonalization.
\newblock In: Signal Processing Conference, 2004 12th European, pp. 209--212.
  IEEE, Washinton, DC (2004)

\bibitem{afsari2008sensitivity}
Afsari, B.: Sensitivity analysis for the problem of matrix joint
  diagonalization.
\newblock SIAM J. Matrix Anal. Appl. \textbf{30}(3), 1148--1171 (2008)

\bibitem{arthur2007k}
Arthur, D., Vassilvitskii, S.: k-means++: The advantages of careful seeding.
\newblock In: Proceedings of the eighteenth annual ACM-SIAM symposium on
  Discrete algorithms, pp. 1027--1035. Society for Industrial and Applied
  Mathematics (2007)

\bibitem{bai2009exploiting}
Bai, Y., de~Klerk, E., Pasechnik, D., Sotirov, R.: Exploiting group symmetry in
  truss topology optimization.
\newblock Optim. Engrg. \textbf{10}(3), 331--349 (2009)

\bibitem{belouchrani1997blind}
Belouchrani, A., Abed-Meraim, K., Cardoso, J.F., Moulines, E.: A blind source
  separation technique using second-order statistics.
\newblock IEEE Trans. Signal Process. \textbf{45}(2), 434--444 (1997).
\newblock SOBI

\bibitem{cai2017perturbation}
Cai, Y., Li, R.C.: Perturbation analysis for matrix joint block
  diagonalization.
\newblock arXiv:1703.00591  (2017)

\bibitem{cai2017algebraic}
Cai, Y., Liu, C.: An algebraic approach to nonorthogonal general joint block
  diagonalization.
\newblock SIAM J. Matrix Anal. Appl. \textbf{38}(1), 50--71 (2017)

\bibitem{cai2015matrix}
Cai, Y., Shi, D., Xu, S.: A matrix polynomial spectral approach for general
  joint block diagonalization.
\newblock SIAM J. Matrix Anal. Appl. \textbf{36}(2), 839--863 (2015)

\bibitem{cardoso1998multidimensional}
Cardoso, J.F.: Multidimensional independent component analysis.
\newblock In: Acoustics, Speech and Signal Processing, 1998. Proceedings of the
  1998 IEEE International Conference on, vol.~4, pp. 1941--1944. IEEE,
  Washinton, DC (1998)

\bibitem{cardoso1993blind}
Cardoso, J.F., Souloumiac, A.: Blind beamforming for non-{G}aussian signals.
\newblock In: IEE Proceedings F (Radar and Signal Processing), vol. 140, pp.
  362--370. IET (1993)

\bibitem{chabriel2014joint}
Chabriel, G., Kleinsteuber, M., Moreau, E., Shen, H., Tichavsky, P., Yeredor,
  A.: Joint matrices decompositions and blind source separation: A survey of
  methods, identification, and applications.
\newblock IEEE Signal Process. Mag. \textbf{31}(3), 34--43 (2014)

\bibitem{cherrak2013non}
Cherrak, O., Ghennioui, H., Abarkan, E.H., Thirion-Moreau, N.: Non-unitary
  joint block diagonalization of matrices using a levenberg-marquardt
  algorithm.
\newblock In: Signal Processing Conference (EUSIPCO), 2013 Proceedings of the
  21st European, pp. 1--5. IEEE, Washinton, DC (2013)

\bibitem{choi2005blind}
Choi, S., Cichocki, A., Park, H.M., Lee, S.Y.: Blind source separation and
  independent component analysis: A review.
\newblock Neural Information Processing-Letters and Reviews \textbf{6}(1)
  (2005)

\bibitem{comon2010handbook}
Comon, P., Jutten, C.: Handbook of Blind Source Separation: Independent
  component analysis and applications.
\newblock Academic press (2010)

\bibitem{de2007reduction}
De~Klerk, E., Pasechnik, D.V., Schrijver, A.: Reduction of symmetric
  semidefinite programs using the regular $\ast$-representation.
\newblock Math. Program. \textbf{109}(2-3), 613--624 (2007)

\bibitem{de2010exploiting}
De~Klerk, E., Sotirov, R.: Exploiting group symmetry in semidefinite
  programming relaxations of the quadratic assignment problem.
\newblock Math. Program. \textbf{122}(2), 225--246 (2010)

\bibitem{de2008decompositions}
De~Lathauwer, L.: Decompositions of a higher-order tensor in block terms--part
  {I}: Lemmas for partitioned matrices.
\newblock SIAM J. Matrix Anal. Appl. \textbf{30}(3), 1022--1032 (2008)

\bibitem{de2008decompositions2}
De~Lathauwer, L.: Decompositions of a higher-order tensor in block terms--part
  {II}: Definitions and uniqueness.
\newblock SIAM J. Matrix Anal. Appl. \textbf{30}(3), 1033--1066 (2008)

\bibitem{de2009survey}
De~Lathauwer, L.: A survey of tensor methods.
\newblock In: 2009 IEEE International Symposium on Circuits and Systems, pp.
  2773--2776. IEEE (2009)

\bibitem{de2000fetal}
De~Lathauwer, L., De~Moor, B., Vandewalle, J.: Fetal electrocardiogram
  extraction by blind source subspace separation.
\newblock IEEE Trans. Biomedical Engrg. \textbf{47}(5), 567--572 (2000)

\bibitem{de2008decompositions3}
De~Lathauwer, L., Nion, D.: Decompositions of a higher-order tensor in block
  terms--part {III}: Alternating least squares algorithms.
\newblock SIAM J. Matrix Anal. Appl. \textbf{30}(3), 1067--1083 (2008)

\bibitem{gatermann2004symmetry}
Gatermann, K., Parrilo, P.A.: Symmetry groups, semidefinite programs, and sums
  of squares.
\newblock J. Pure Appl. Algebra \textbf{192}(1), 95--128 (2004)

\bibitem{higham2006symmetric}
Higham, N.J., Mackey, D.S., Mackey, N., Tisseur, F.: Symmetric linearizations
  for matrix polynomials.
\newblock SIAM J. Matrix Anal. Appl. \textbf{29}(1), 143--159 (2006)

\bibitem{hyvarinen2004independent}
Hyv{\"a}rinen, A., Karhunen, J., Oja, E.: Independent component analysis,
  vol.~46.
\newblock John Wiley \& Sons (2004)

\bibitem{de2011numerical}
de~Klerk, E., Dobre, C., \.{P}asechnik, D.V.: Numerical block diagonalization
  of matrix $\ast$-algebras with application to semidefinite programming.
\newblock Math. Program. \textbf{129}(1), 91--111 (2011)

\bibitem{mackey2006vector}
Mackey, D.S., Mackey, N., Mehl, C., Mehrmann, V.: Vector spaces of
  linearizations for matrix polynomials.
\newblock SIAM J. Matrix Anal. Appl. \textbf{28}(4), 971--1004 (2006)

\bibitem{macqueen1967some}
MacQueen, J., et~al.: Some methods for classification and analysis of
  multivariate observations.
\newblock In: Proceedings of the fifth Berkeley symposium on mathematical
  statistics and probability, vol.~1, pp. 281--297. Oakland, CA, USA. (1967)

\bibitem{maehara2010numerical}
Maehara, T., Murota, K.: A numerical algorithm for block-diagonal decomposition
  of matrix $*$-algebras with general irreducible components.
\newblock Japan J. Indust. Appl. Math. \textbf{27}(2), 263--293 (2010)

\bibitem{maehara2011algorithm}
Maehara, T., Murota, K.: Algorithm for error-controlled simultaneous
  block-diagonalization of matrices.
\newblock SIAM J. Matrix Anal. Appl. \textbf{32}(2), 605--620 (2011)

\bibitem{mehrmann2002polynomial}
Mehrmann, V., Watkins, D.: Polynomial eigenvalue problems with {H}amiltonian
  structure.
\newblock Electron. Trans. Numer. Anal \textbf{13}, 106--118 (2002)

\bibitem{moreau2001generalization}
Moreau, E.: A generalization of joint-diagonalization criteria for source
  separation.
\newblock IEEE Trans. Signal Process. \textbf{49}(3), 530--541 (2001)

\bibitem{murota2010numerical}
Murota, K., Kanno, Y., Kojima, M., Kojima, S.: A numerical algorithm for
  block-diagonal decomposition of matrix $*$-algebras with application to
  semidefinite programming.
\newblock Japan J. Indust. Appl. Math. \textbf{27}(1), 125--160 (2010)

\bibitem{yuji2016eigenvector}
Nakatsukasa, Y., Tisseur, F.: Eigenvector error bound and perturbation for
  polynomial and rational eigenvalue problems

\bibitem{nion2011tensor}
Nion, D.: A tensor framework for nonunitary joint block diagonalization.
\newblock IEEE Trans. Signal Process. \textbf{59}(10), 4585--4594 (2011)

\bibitem{pereira2003solvents}
Pereira, E.: On solvents of matrix polynomials.
\newblock Appl. Numer. Math. \textbf{47}(2), 197--208 (2003)

\bibitem{shi2015some}
Shi, D., Cai, Y., Xu, S.: Some perturbation results for a normalized
  non-orthogonal joint diagonalization problem.
\newblock Linear Algebra Appl. \textbf{484}, 457--476 (2015)

\bibitem{theis2004new}
Theis, F.J.: A new concept for separability problems in blind source
  separation.
\newblock Neural Comput. \textbf{16}(9), 1827--1850 (2004)

\bibitem{theis2005blind}
Theis, F.J.: Blind signal separation into groups of dependent signals using
  joint block diagonalization.
\newblock In: Circuits and Systems, 2005. ISCAS 2005. IEEE International
  Symposium on, pp. 5878--5881. IEEE (2005)

\bibitem{theis2006towards}
Theis, F.J.: Towards a general independent subspace analysis.
\newblock In: Advances in Neural Information Processing Systems, pp.
  1361--1368. MIT Press, Cambridge, MA (2006)

\bibitem{tichavsky2017non}
Tichavsky, P., Phan, A.H., Cichocki, A.: Non-orthogonal tensor diagonalization.
\newblock Signal Process. \textbf{138}, 313 -- 320 (2017)

\bibitem{von2007tutorial}
Von~Luxburg, U.: A tutorial on spectral clustering.
\newblock Statistics and computing \textbf{17}(4), 395--416 (2007)

\bibitem{yeredor2000blind}
Yeredor, A.: Blind source separation via the second characteristic function.
\newblock Signal Process. \textbf{80}(5), 897--902 (2000)

\end{thebibliography}
\end{document}